\documentclass[11pt]{article}

\usepackage{amscd}
\usepackage{amsmath}
\usepackage{amssymb}
\usepackage{amstext}
\usepackage{amsthm}
\usepackage{bbold}
\usepackage{bm}
\usepackage{booktabs}
\usepackage{color}
\usepackage{colonequals}
\usepackage{comment}

\usepackage{easybmat}
\usepackage{etex}
\usepackage{enumitem}
\usepackage{framed}
\usepackage{authblk}
\usepackage[dvips,letterpaper,margin=1in]{geometry}
\usepackage{graphicx}
\usepackage{listings}
\usepackage{longtable}
\usepackage{mathtools}
\usepackage{multirow}
\usepackage{rotating}
\usepackage{setspace}
\usepackage{siunitx}
\usepackage{tabu}
\usepackage{verbatim}

\definecolor{cblack}{rgb}{0,0,0}
\definecolor{cblue}{rgb}{0.121569,0.466667,0.705882}    
\definecolor{corange}{rgb}{1.000000,0.498039,0.054902}  
\definecolor{cgreen}{rgb}{0.172549,0.627451,0.172549}   
\definecolor{cred}{rgb}{0.839216,0.152941,0.156863}     
\definecolor{cpurple}{rgb}{0.580392,0.403922,0.741176}  
\definecolor{cbrown}{rgb}{0.549020,0.337255,0.294118}   
\definecolor{cpink}{rgb}{0.890196,0.466667,0.760784}
\definecolor{cgray}{rgb}{0.498039,0.498039,0.498039}
\definecolor{cgreen2}{rgb}{0.7372549019607844, 0.7411764705882353, 0.13333333333333333}

\usepackage{hyperref}
\hypersetup{
  linkcolor  = cblue,
  citecolor  = cgreen,
  urlcolor   = corange,
  colorlinks = true,
}

\newtheorem{theorem}{Theorem}[section]
\newtheorem{remark}[theorem]{Remark}
\newtheorem{assumption}[theorem]{Assumption}
\newtheorem{lemma}[theorem]{Lemma}

\newtheorem{definition}[theorem]{Definition}
\newtheorem{example}[theorem]{Example}
\newtheorem{proposition}[theorem]{Proposition}
\newtheorem{corollary}[theorem]{Corollary}
\newtheorem{conjecture}[theorem]{Conjecture}
\theoremstyle{plain} 
\newcommand{\thistheoremname}{}
\newtheorem*{genericthm}{\thistheoremname}

\newcommand{\Cramer}{Cram\'{e}r}

\newcommand{\Erdos}{Erd\H{o}s}

\newcommand{\Mobius}{M\"{o}bius}

\newcommand{\Renyi}{R\'{e}nyi}

\newcommand{\what}{\widehat}

\renewcommand{\emptyset}{\varnothing}

\makeatletter
\def\moverlay{\mathpalette\mov@rlay}
\def\mov@rlay#1#2{\leavevmode\vtop{%
   \baselineskip\z@skip \lineskiplimit-\maxdimen
   \ialign{\hfil$\m@th#1##$\hfil\cr#2\crcr}}}
\newcommand{\charfusion}[3][\mathord]{
    #1{\ifx#1\mathop\vphantom{#2}\fi
        \mathpalette\mov@rlay{#2\cr#3}
      }
    \ifx#1\mathop\expandafter\displaylimits\fi}
\makeatother

\newcommand{\EE}{\mathbb{E}}

\newcommand{\NN}{\mathbb{N}}
\newcommand{\PP}{\mathbb{P}}
\newcommand{\QQ}{\mathbb{Q}}
\newcommand{\RR}{\mathbb{R}}

\newcommand{\ZZ}{\mathbb{Z}}
\DeclareSymbolFont{bbold}{U}{bbold}{m}{n}
\DeclareSymbolFontAlphabet{\mathbbold}{bbold}
\newcommand{\One}{\mathbbold{1}}

\newcommand{\bt}{\bm t}

\newcommand{\bx}{\bm x}
\newcommand{\by}{\bm y}
\newcommand{\bz}{\bm z}

\newcommand{\bF}{\bm F}

\newcommand{\bW}{\bm W}
\newcommand{\bX}{\bm X}
\newcommand{\bY}{\bm Y}

\newcommand{\one}{\bm{1}}

\newcommand{\sC}{\mathcal{C}}
\newcommand{\sD}{\mathcal{D}}

\newcommand{\sN}{\mathcal{N}}

\newcommand{\sP}{\mathcal{P}}
\newcommand{\sQ}{\mathcal{Q}}

\newcommand{\sX}{\mathcal{X}}

\DeclareSymbolFont{sfoperators}{OT1}{cmss}{m}{n}
\DeclareSymbolFontAlphabet{\mathsf}{sfoperators}
\makeatletter
\renewcommand{\operator@font}{\mathgroup\symsfoperators}
\makeatother

\DeclareMathOperator{\poly}{poly}
\DeclareMathOperator{\Tr}{Tr}

\DeclareMathOperator{\sgn}{sgn}

\DeclareMathOperator{\Unif}{Unif}

\DeclareMathOperator{\Id}{Id}
\DeclareMathOperator{\Var}{Var}

\newcommand{\Ex}{\mathop{\mathbb{E}}}  

\newcommand{\Ber}{\mathsf{Ber}}
\newcommand{\cdeg}{\mathsf{cdeg}}
\newcommand{\Avg}{\mathsf{Avg}}
\newcommand{\Adv}{\mathsf{Adv}}
\newcommand{\CAdv}{\mathsf{CAdv}}

\newcommand{\Univ}{\mathsf{Univ}}
\newcommand{\test}{\mathsf{test}}
\newcommand{\dbar}{\,\|\,}
\newcommand{\eff}{\mathsf{eff}}
\renewcommand{\epsilon}{\varepsilon}

\newcommand\numberthis{\addtocounter{equation}{1}\tag{\theequation}}

\title{Low coordinate degree algorithms I: Universality of computational thresholds for hypothesis testing}
\date{March 12, 2024}

\usepackage{authblk}
\author{Dmitriy Kunisky\thanks{Email: \textit{dmitriy.kunisky@yale.edu}. Partially supported by ONR Award N00014-20-1-2335, a Simons Investigator Award to Daniel Spielman, and NSF grants DMS-1712730 and DMS-1719545.}}
\affil{Department of Computer Science, Yale University}

\begin{document}

\maketitle

\thispagestyle{empty}

\begin{abstract}
    We study when \emph{low coordinate degree functions (LCDF)}---linear combinations of functions depending on small subsets of entries of a vector---can hypothesis test between high-dimensional probability measures.
    These functions are a generalization, proposed in Hopkins' 2018 thesis but seldom studied since, of \emph{low degree polynomials (LDP)}, a class widely used in recent literature as a proxy for all efficient algorithms for tasks in statistics and optimization.
    Instead of the orthogonal polynomial decompositions used in LDP calculations, our analysis of LCDF is based on the \emph{Efron-Stein} or \emph{ANOVA decomposition}, making it much more broadly applicable.
    By way of illustration, we prove \emph{channel universality} for the success of LCDF in testing for the presence of sufficiently ``dilute'' random signals through noisy channels: the efficacy of LCDF depends on the channel only through the scalar \emph{Fisher information} for a class of channels including nearly arbitrary additive i.i.d.\ noise and nearly arbitrary exponential families.
    As applications, we extend lower bounds against LDP for spiked matrix and tensor models under additive Gaussian noise to lower bounds against LCDF under general noisy channels.
    We also give a simple and unified treatment of the effect of \emph{censoring} models by erasing observations at random and of \emph{quantizing} models by taking the sign of the observations.
    These results are the first computational lower bounds against any large class of algorithms for all of these models when the channel is not one of a few special cases, and thereby give the first substantial evidence for the universality of several statistical-to-computational gaps.
\end{abstract}

\clearpage

\tableofcontents

\pagestyle{empty}

\clearpage

\setcounter{page}{1}
\pagestyle{plain}

\section{Introduction}

An effective theory of high-dimensional statistics requires computational as well as statistical considerations---in working with large datasets, it matters not only whether it is \emph{possible} to solve a problem with \emph{some} computation, but also whether that computation is efficient.
A rich literature has grown around understanding the difference between these two demands, including \emph{statistical-to-computational gaps} between the parameter regimes where a problem can be solved at all and those where it can be solved efficiently.
Such gaps are conjectured to occur in problems including finding weak community structure \cite{DKMZ-2011-SBM, DKMZ-2011-AsymptoticAnalysisSBM, Moore-2017-SBMReview, Abbe-2017-SBMReview} and planted cliques \cite{Jerrum-1992-LargeCliques, MPW-2015-PlantedClique, BHKKMP-2019-PlantedClique} in random graphs, various forms of principal component analysis for matrices \cite{BR-2013-SparsePCA, LKZ-2015-PhaseTransitionsSparsePCA, LKZ-2015-LowRankChannelUniversality, BMVVX-2018-InfoTheoretic, PWBM-2018-PCAI, AKJ-2018-SpikedWigner} and tensors \cite{RM-2014-TensorPCA, LMLKZ-2017-TransitionsSpikedTensor, BAGJ-2018-TensorPCALocal, JLM-2020-StatisticalTensorPCA}, and numerous other problems \cite{AMKMZ-2018-CommitteeMachineNeuralNetwork,COGHWZ-2022-PhaseTransitionsGroupTesting,KLLM-2022-StatCompGapReinforcementLearning}.
Unfortunately, it seems out of reach of current techniques to show that such problems are computationally hard under standard complexity-theoretic assumptions like $\mathsf{P} \neq \mathsf{NP}$.
Instead, many other forms of evidence of computational hardness have been proposed, including the analysis of various specific classes of algorithms (such as convex optimization and the sum-of-squares hierarchy \cite{BHKKMP-2019-PlantedClique, HKPRSS-2017-SOSSpectral, RSS-2018-EstimationSOS}, Markov chain Monte Carlo \cite{Jerrum-1992-LargeCliques, DFJ-2002-IndependentSetMCMC, CMZ-2023-LinearPlantedCliquesMetropolis, GJX-2023-PlantedCliqueMCMC}, and message-passing algorithms inspired by statistical physics \cite{ZK-2016-Review}), the study of the geometry of solution spaces and optimization landscapes \cite{ACO-2008-AlgorithmicBarriers, IKKM-2012-XORSAT, GS-2014-LimitsLocal, GM-2017-LandscapeTensorDecomposition, BAMMN-2017-LandscapeSpikedTensor, GZ-2019-LandscapePlantedClique}, and reductions among average-case problems \cite{BR-2013-SparsePCA, MW-2015-ReductionsSubmatrixDetection, BBH-2018-ReducibilityPlantedSparse, BB-2019-ReductionsSparsePCA, BB-2020-ReducibilityStatCompGaps}.

One class of computations that has come to play a central role in this pursuit is \emph{low degree polynomial (LDP)} algorithms.
LDP are simple to describe---as the name suggests, they just compute polynomials of the given data to solve a problem---and convenient to analyze, and yet seem just as powerful as other, more complicated algorithms for many settings of hypothesis testing \cite{BHKKMP-2019-PlantedClique, HKPRSS-2017-SOSSpectral, Hopkins-2018-Thesis, BKW-2019-ConstrainedPCA, KWB-2022-LowDegreeNotes, DKWB-2019-SubexponentialTimeSparsePCA, BBKMW-2020-SpectralPlantingColoring, BAHSWZ-2022-FranzParisiLowDegree, BBHLS-2020-SQLowDegree}, estimation \cite{HS-2017-BayesianEstimation, SW-2020-LowDegreeEstimation,MW-2022-AMPLowDegree}, optimization \cite{GJW-2020-LowDegreeOptimization, Wein-2020-LowDegreeIndependentSet}, and constraint satisfaction \cite{BH-2022-LowDegreeKSAT}.
However, analyzing LDP also comes with some drawbacks.
First, its scope is currently limited to data with convenient probability distributions that admit explicit families of orthogonal polynomials, which give a well-behaved basis in which to consider arbitrary LDP algorithms.
In particular, the vast majority of LDP results concern just two data distributions: Gaussian and Bernoulli.
Second, LDP analyses typically involve repetitive combinatorial calculations (arising from high-dimensional orthogonal polynomial expansions), which address computational hardness on a case-by-case basis but arguably do not grant much insight into what aspects of a problem make it costly to solve.
The LDP framework so far has yielded a large zoo of examples of hard problems and statistical-to-computational gaps, but not yet a theory to draw quick and accurate conclusions about questions like: how will this problem change if we are provided with side information? If we only make a fraction of our observations? If our data are corrupted by a different kind of noise?

In this paper, we develop tools to work with the more general class of \emph{low coordinate degree function (LCDF)} algorithms.
We will see that LCDF are amenable to a much more general theory than LDP.
In particular, we give an analysis of the performance of LCDF for a class of hypothesis testing tasks under essentially arbitrary noise models, and take initial steps towards developing rules for how computational hardness changes under simple modifications of a statistical model.

\subsection{Detecting structured latent variables}

Many models of hypothesis testing exhibiting statistical-to-computational gaps may be cast in the following general way; we will phrase all of our main results in these terms.
\begin{definition}[Continuous latent variable model]
    \label{def:clvm}
    Let $\Sigma \subseteq \RR$ contain zero, and $N \geq 1$.
    Let $\sX$ be a probability measure over $\Sigma^N$, and, for each $x \in \Sigma$, let $\sP_x$ be a probability measure over a measurable space $\Omega$.
    A \emph{continuous\footnote{The term ``continuous'' here refers to $\Sigma$ consisting of real numbers. A companion paper will examine the complementary discrete setting.} latent variable model (CLVM)} specified by these objects $(\sX, \sP)$ consists of the following pair of probability measures over $\Omega^N$:
    \begin{enumerate}
    \item Sample $\by \sim \QQ$ by sampling $y_i \sim \sP_0$ for each $i \in [N]$ independently.
    \item Sample $\by \sim \PP$ by first sampling $\bx \sim \sX$. Then, sample $y_i \sim \sP_{x_i}$ for each $i \in [N]$ independently.
    \end{enumerate}
\end{definition}
\noindent
The \emph{latent variable} is $\bx$ that is observed only indirectly through $\by$ under $\PP$.
Often one is interested in estimating $\bx$; however, we will focus here on the task of hypothesis testing between $\QQ$ and $\PP$, i.e., between $\bx = \bm 0$ and $\bx \sim \sX$, which we call \emph{detection} of the \emph{signal} $\bx \sim \sX$.
We work with the following specific notion.
\begin{definition}[Strong detection]
    Consider pairs of probability measures $\PP_n, \QQ_n$ over measurable spaces $\Omega_n$.
    We say that functions $\test_n: \Omega_n \to \{\texttt{p}, \texttt{q}\}$ achieve \emph{strong detection} if
    \begin{equation}
        \lim_{n \to \infty} \PP_n[\test_n(\by) = \texttt{q}] = \lim_{n \to \infty} \QQ_n[\test_n(\by) = \texttt{p}] = 0,
    \end{equation}
    that is, if the sequence of hypothesis tests $\test_n$ have both Type~I and Type~II error probabilities tending to zero.
    When this asymptotic setting is clear from context, we informally call the entire sequence $(\test_n)$ ``a test.''
\end{definition}
\noindent
We refer to the possibility of strong detection by an arbitrary test as \emph{statistical} feasibility of such a problem, and to the possibility of strong detection by an efficiently computable test (meaning in polynomial time in $n$, unless otherwise specified) as \emph{computational} feasibility.

We call $\sX$ the \emph{prior} (on $\bx$) and the collection of measures $\sP = (\sP_x)_{x \in \Sigma}$ the \emph{channel} through which we observe $\bx$.\footnote{Often in related literature the channel is written as a conditional probability $\sP(y \mid x)$, but our notation $\sP_x$ will be more convenient later to speak about the individual measures $\sP_x$.}
The major assumption of such a model is that $\QQ$ is a product measure and $\PP$ is a product measure after conditioning on the latent variable $\bx$.

\begin{remark}[Signal-to-noise parameters]
    \label{rem:snr}
Many models of interest also include an explicit \emph{signal-to-noise (SNR)} parameter.
This is a scalar $\lambda > 0$ that enters linearly into $\sX = \sX_{\lambda}$, so that there is some base measure $\widetilde{\sX}$ and we sample $\bx \sim \sX_{\lambda}$ by drawing $\widetilde{\bx} \sim \widetilde{\sX}$ and setting $\bx = \lambda \widetilde{\bx}$.
We abbreviate this relationsip of measures $\sX = \lambda \cdot \widetilde{\sX}$.
One then looks, over a growing sequence of priors $\widetilde{\sX}_n$ and resulting $\sX_{n, \lambda} = \lambda \cdot \widetilde{\sX}_n$, at how $\lambda = \lambda(n)$ must scale for detection to become possible as $n \to \infty$, either statistically or computationally.
\end{remark}

Two important channels, those by far most widely studied in the literature, give the following special CLVMs:
\begin{itemize}
\item \emph{Additive} or \emph{Wigner Gaussian models}, where $\Sigma = \Omega = \RR$ and $\sP_x = \sN(x, \sigma^2)$ for some $\sigma^2 > 0$.
    Under $\QQ$ we observe an i.i.d.\ Gaussian vector $\by \sim \sN(\bm 0, \sigma^2\bm I_N)$, while under $\PP$ we observe $\by = \bx + \bz$ for $\bx \sim \sX$ and $\bz \sim \sN(\bm 0, \sigma^2\bm I_N)$.
\item \emph{Bernoulli models}, where $\Sigma = [-c, 1 - c]$, $\Omega = \{0, 1\}$, and $\sP_x = \Ber(c + x)$ for some $c \in (0, 1)$.
    Under $\QQ$ we observe an i.i.d.\ vector of i.i.d.\ Bernoulli random variables $y_i \sim \Ber(c)$, while under $\PP$ we draw $\bx \sim \sX$ and observe biased Bernoulli variables $y_i \sim \Ber(c + x_i)$.
\end{itemize}
For instance, spiked (Wigner, not Wishart) matrix and tensor models are usually formulated as additive Gaussian models, and random graph models like the planted clique, planted dense subgraph, and stochastic block models as Bernoulli models.

These are not the only interesting channels: other work has considered exponential \cite{MMX-2019-PlantedMatching}, folded Gaussian \cite{CKKVZ-2010-InferenceParticleTracking}, and Poisson \cite{KN-2011-SBM} observations, for example, as well as general additive noise for continuous signals \cite{LKZ-2015-LowRankChannelUniversality, KXZ-2016-MutualInformationChannelUniversality, PWBM-2018-PCAI,MRY-2018-AdaptiveMatrixDenoising} and general observation models for discrete signals \cite{LMX-2015-LabelledSBM, BBH-2019-UniversalityComputationalSubmatrixDetection, DWXY-2021-PlantedMatchingInfiniteOrder}.
Yet, some of our most useful tools are restricted to Gaussian and Bernoulli channels.
We next review some of these tools for understanding the computational power of LDP for detection in CLVMs through these special channels.

\subsection{Low degree algorithms}

The simple yet powerful class of \emph{low degree polynomial~(LDP)} algorithms seek to solve a detection problem by computing a low degree polynomial of $\by$ whose typical value is different under $\by \sim \PP$ and $\by \sim \QQ$.
The following quantity, which we follow some recent work like \cite{RSWY-2022-CountCommunitiesLowDegree} in calling the \emph{advantage}, gives a simple measurement of how well LDP can ``separate'' $\PP$ and $\QQ$.\footnote{A perhaps more natural notion of separation, treating $\PP$ and $\QQ$ symmetrically, is given in \cite{BAHSWZ-2022-FranzParisiLowDegree,RSWY-2022-CountCommunitiesLowDegree}. As detailed there, bounding the advantage is a means of excluding this kind of separation. We could apply this alternative framing to LCDF as well, but we do not pursue that direction here.}

\begin{equation}
    \Adv_{\leq D}(\sX, \sP) \colonequals \left\{\begin{array}{ll} \text{maximize} & \Ex_{\by \sim \PP} f(\by) \\ \text{subject to} & \Ex_{\by \sim \QQ} f(\by)^2 \leq 1, \\ & f \in \RR[\by]_{\leq D} \end{array}\right\}. \label{eq:advantage}
\end{equation}

The following rough conjecture, which originates in a line of work on sum-of-squares optimization \cite{BHKKMP-2019-PlantedClique, HS-2017-BayesianEstimation, HKPRSS-2017-SOSSpectral, Hopkins-2018-Thesis}, expresses the idea that polynomials are optimal efficient tests for a given runtime budget.
It is imprecise, and various pathological cases have been devised to refute the conjecture without some additional clauses \cite{HW-2021-CounterexamplesLowDegree, Kunisky-2020-LowDegreeMorris, KM-2021-ReconstructionTreesLowDegree, DK-2022-ComponentAnalysisLatticeBasis, ZSWB-2022-LatticeMethodsSOS}, but for many problems it gives predictions of hardness that match all other evidence.
\begin{conjecture}[Informal low degree conjecture]
    \label{conj:low-deg}
    Consider a sequence of CLVMs $(\sX_n, \sP)$, giving rise to sequences of measures $\PP_n$ and $\QQ_n$ on $\Omega^N$ for some $N = N(n)$.
    For ``sufficiently nice'' priors and channels (see Chapter 2 of \cite{Hopkins-2018-Thesis} for some discussion of how these conditions might be formalized), we conjecture:
    \begin{enumerate}
    \item If, for some $D = D(n) = \omega(\log n)$, $\Adv_{\leq D}(\sX_n, \sP) = O(1)$ as $n \to \infty$, then there is no test computable in time polynomial in $n$ that achieves strong detection between $\PP_n$ and $\QQ_n$.
    \item If, for some $D = D(n) = \omega(1)$, $\Adv_{\leq D}(\sX_n, \sP) = O(1)$ as $n \to \infty$, then there is no test computable in time $\exp(D / \mathsf{polylog}(n))$ that achieves strong detection between $\PP_n$ and $\QQ_n$.
    \end{enumerate}
\end{conjecture}
\noindent
The second case is meant to allow for the study of subexponential time algorithms when we take $D = n^{\delta}$ for some $\delta \in (0, 1)$ (see, e.g., \cite{DKWB-2019-SubexponentialTimeSparsePCA}).

Conversely, the \emph{unboundedness} of the advantage also seems to have computational consequences: if, e.g., for $D = (\log n)^{K}$ we have $\Adv_{\leq D}(\sX_n, \sP) = \omega(1)$, then we expect there to be an algorithm running in quasipolynomial time $n^D = \exp(\mathsf{polylog}(n))$ that distinguishes $\PP$ from $\QQ$, namely, computing and thresholding the polynomial that is the optimizer in the advantage \eqref{eq:advantage}.
This, too, is not always true, but for ``nice'' problems and when we are considering $D = \mathsf{polylog}(n)$ it does seem to be accurate; see Section 4 of \cite{KWB-2022-LowDegreeNotes} for discussion.
Thus, technicalities aside, the bottom line that we will take as our motivation going forward is that \textbf{the boundedness or divergence of the advantage appears to govern the computational hardness of strong detection.}

The advantage may be analyzed quite directly in some models, because the optimal $f$ may be characterized as the orthogonal projection in $L^2(\QQ)$ of the likelihood ratio $d\PP / d\QQ$ to the subspace of polynomials of degree at most $D$, the \emph{low degree likelihood ratio (LDLR)} (see, e.g., Proposition~1.15 of \cite{KWB-2022-LowDegreeNotes}).
The advantage itself is the norm in $L^2(\QQ)$ of the LDLR.
Beyond that general characterization, we also have the following elegant computation specifically for additive Gaussian models.

\begin{proposition}[Additive Gaussian advantage; Theorem 2.6 of \cite{KWB-2022-LowDegreeNotes}]
    \label{prop:gaussian-additive}
    For any prior $\sX$ and $\sigma^2 > 0$,
    \begin{equation}
        \Adv_{\leq D}(\sX, \sN(x, \sigma^2))^2 = \Ex_{\bx^{(1)}, \bx^{(2)} \sim \sX} \exp^{\leq D}\left(\frac{1}{\sigma^2} \langle \bx^{(1)}, \bx^{(2)} \rangle\right),
        \label{eq:gauss-adv}
    \end{equation}
    where $\bx^1$ and $\bx^2$ are independent draws from $\sX$ and $\exp^{\leq D}$ is the truncated Taylor series of the exponential function,
    \begin{equation}
        \exp^{\leq D}(t) \colonequals \sum_{d = 0}^D \frac{t^d}{d!}.
    \end{equation}
\end{proposition}
\noindent
The proof explicitly decomposes the LDLR in orthogonal polynomials in $L^2(\QQ)$, which, since $\QQ$ is a Gaussian product measure, are the multivariate Hermite polynomials.

This formula is remarkable in how nicely the roles of the three parameters $D, \sX$, and $\sP$ separate: $D$ determines the truncation of the exponential power series, $\sX$ determines the distribution of $\langle \bx^{(1)}, \bx^{(2)} \rangle$, and $\sP$ determines the variance $\sigma^2$.
Especially remarkable is that the advantage---\emph{a priori} a function of the high-dimensional measure $\sX$---only depends on the distribution of the \emph{scalar} inner product $\langle \bx^{(1)}, \bx^{(2)}\rangle$, for which we borrow from statistical physics the term \emph{overlap}.
This dramatically reduces the dimensionality of the computation of the advantage, which is both formally intriguing and practically useful for proving LDP lower bounds and other results (see, e.g., the arguments of \cite{BKW-2019-ConstrainedPCA,KWB-2022-LowDegreeNotes,BBKMW-2020-SpectralPlantingColoring, BAHSWZ-2022-FranzParisiLowDegree}).
We might view the Gaussian additive channel $\sN(x, \sigma^2)$ as leading to an \emph{integrable} CLVM, where we can give a compact closed form for the advantage using the special structure of the Hermite polynomials.

A few further channels, including Bernoulli, have been found to lead to likewise integrable CLVMs which can be directly compared to Gaussian models (see Section~\ref{sec:related}).
These findings, however, are still very brittle: with these tools we still cannot say anything about the performance of LDP, for example, for a channel applying additive noise with a density given by a slight perturbation of the Gaussian density, though it is intuitively obvious that this should not change the computational cost of detection very much.

\subsection{Low coordinate degree algorithms}

We now introduce the \emph{low coordinate degree functions (LCDF)} that will generalize LDP.
In words, while LDP are linear combinations of low degree monomials---products of entries in a small number of coordinates of a vector---LCDF are linear combinations of \emph{arbitrary functions} of entries in a small number of coordinates.

For $\by \in \Omega^N$ and $T \subseteq [N]$, we write $\by_T \in \Omega^T$ for the restriction of $\by$ to the coordinates in $T$.
We define subspaces of $L^2(\QQ)$
\begin{align}
  V_T &\colonequals \{f \in L^2(\QQ): f(\by) \text{ depends only on } \by_T\}, \\
        V_{\leq D} &\colonequals \sum_{\substack{T \subseteq [N] \\ |T| \leq D}} V_T.
\end{align}
Here a sum of subspaces $V + W$ is the set of all $v + w$ with $v \in V$ and $w \in W$. Note that $V_{[N]} = V_{\leq N} = L^2(\QQ)$.
\begin{definition}[Coordinate degree]
    For $f \in L^2(\QQ)$, $\cdeg(f) \colonequals \min\{D: f \in V_{\leq D}\}$.
\end{definition}
\noindent
For $f$ a polynomial we have $\deg(f) \leq \cdeg(f)$, so $\RR[\by]_{\leq D} \subseteq V_{\leq D}$, and LCDF are indeed a larger class than LDP for a given degree bound $D$.
We note also that, unlike polynomial degree, all $L^2$ functions have a finite coordinate degree.

This broader class was proposed in Hopkins' thesis \cite{Hopkins-2018-Thesis}, where it was suggested as a more natural collection of statistics than LDP, in particular one that is closed under arbitrary entrywise transformations.
Yet, to the best of our knowledge, LCDF have only been studied a few times subsequently in some special cases \cite{BBHLS-2020-SQLowDegree, KM-2021-ReconstructionTreesLowDegree, HM-2024-LowDegreeBroadcastingTrees}.

Paralleling the LDP framework, we define the \emph{coordinate advantage}
\begin{equation}
    \CAdv_{\leq D}(\sX, \sP) \colonequals \left\{\begin{array}{ll} \text{maximize} & \Ex_{\by \sim \PP} f(\by) \\ \text{subject to} & \Ex_{\by \sim \QQ} f(\by)^2 \leq 1, \\ & f \in V_{\leq D} \end{array}\right\} \geq \Adv_{\leq D}(\sX, \sP).
\end{equation}
As for LDP, the optimal $f$ is the orthogonal projection of the likelihood ratio to $V_{\leq D}$, which we call the \emph{low coordinate degree likelihood ratio (LCDLR)}, and the coordinate advantage is the norm of this projection.

We call $\Adv_{\leq D}$ the \emph{polynomial advantage} when we want to emphasize the distinction between it and $\CAdv_{\leq D}$.
Because of the above inequality, bounding the coordinate advantage also bounds the polynomial advantage and thus, conditional on Conjecture~\ref{conj:low-deg}, shows computational hardness of strong detection.\footnote{One may of course also formulate a ``low coordinate degree conjecture'' paralleling Conjecture~\ref{conj:low-deg}, for a version of which the reader may consult \cite{Hopkins-2018-Thesis}.}
Alternatively, one may view bounding the coordinate advantage as just a lower bound against LCDF, which are an interesting class of algorithms, at least as powerful as LDP, and perhaps, as Hopkins proposed, more natural.
And conversely, as for LDP, the divergence of the coordinate advantage reasonably suggests that there should be an LCDF-based algorithm achieving strong detection.

\subsection{Summary of contributions}

We present informal summaries of our main results before giving precise statements in Section~\ref{sec:results}.

\paragraph{General theory on universality}
The following summarizes what our general results for arbitrary CLVMs will imply.
We indicate ``in quotation marks'' the vague terms to be clarified later.
\begin{theorem}[Informal]
    \label{thm:informal}
    The coordinate advantage is the same up to constants for a ``sufficiently dilute'' prior observed through any channel having a given Fisher information (a scalar parameter) that either (1) is addition of i.i.d.\ noise with a ``sufficiently nice'' distribution (including any symmetric, positive, and smooth density on $\RR$), or (2) makes observations in a ``sufficiently nice'' exponential family (including that generated by any subgaussian base measure).
\end{theorem}
\noindent
When Theorem~\ref{thm:informal} applies, the coordinate advantage is also close to the polynomial advantage for the additive Gaussian channel, for which we have the formula of Proposition~\ref{prop:gaussian-additive} and numerous prior results.
This existing analysis of additive Gaussian models therefore generalizes through our results ``for free'' to much more general noise models.

\paragraph{Applications}
In this way, we will show the universality of two of the best-known statistical-to-computational gaps in high-dimensional hypothesis testing: those in spiked matrix \cite{Johnstone-2001-LargestEigenvaluePCA} and spiked tensor models \cite{RM-2014-TensorPCA}.
In brief: we consider, for $q \geq 2$, $\lambda = \lambda(n) \geq 0$, and $\bx \in \RR^n$ a suitably scaled random vector, distinguishing a symmetric tensor of i.i.d.\ random variables from one to which the rank one ``signal'' tensor $\lambda \bx^{\otimes q}$ has been added.
In both the matrix case $q = 2$ and the tensor case $q \geq 3$, the computational difficulty of the problem depends on $\lambda$.
In the matrix case, there is a ``sharp'' critical $\lambda_{\mathsf{comp}}$ such that for $\lambda < \lambda_{\mathsf{comp}}$ strong detection is believed to require nearly exponential time in $n$, while for $\lambda > \lambda_{\mathsf{comp}}$ a simple algorithm achieves strong detection.
In the tensor case, there is a ``smoother'' transition between the easy and hard regimes, where there is a range of values of $\lambda_{\mathsf{poly}} < \lambda < \lambda_{\mathsf{exp}}$ for which strong detection is believed to require subexponential time $\exp(O(n^{\delta}))$ for $\delta \in (0, 1)$ depending on $\lambda$.\footnote{In either case there is also a statistical threshold, an even smaller $\lambda_{\mathsf{stat}}$ below which strong detection is impossible.}
In both cases, previous work gave evidence for these claims under Gaussian noise by analyzing LDP \cite{HKPRSS-2017-SOSSpectral, Hopkins-2018-Thesis, KWB-2022-LowDegreeNotes}.
We give a twofold generalization, identifying the computationally hard regimes of $\lambda$ for the broader class of LCDF and for nearly arbitrary additive noise.\footnote{We also discuss in Remarks~\ref{rem:sbm} and \ref{rem:censored-sbm} how suitable choices of exponential family channels yield lower bounds for certain formulations of the stochastic block model.}
Notably, these results follow nearly automatically from our general results applied to the above prior work.

\paragraph{Channel calculus}
Lastly, we begin to develop a ``calculus'' that seeks to understand how general modifications of the noisy channel affect the difficulty of hypothesis testing.
We consider two operations applied after the noisy channel in a CLVM:
\begin{enumerate}
\item In \emph{censorship}, meant to describe ``missing data'' in statistical parlance, each coordinate $y_i$ of our observation is replaced with a null symbol ``$\bullet$'' independently with probability $\eta \in [0, 1]$.
\item In \emph{quantization}, meant to describe observations with low numerical precision, each coordinate $y_i$ of our observation is replaced with $\sgn(y_i)$.
\end{enumerate}
Each may be seen as transforming the channel $\sP$ into a new channel $\sP^{\prime}$.
We compute the Fisher information following either transformation, which, per Theorem~\ref{thm:informal}, describes how the computational cost of strong detection changes.
As a consequence, in the spiked matrix and tensor models, we obtain---again nearly automatically---predictions of computational thresholds with accompanying lower bounds against LCDF under arbitrary noise and when a constant fraction $\eta$ of the observations are withheld or when the observations are quantized.

\subsection{Related work}
\label{sec:related}

\paragraph{Universality in probability theory}
The paradigm of studying a class of models by first thoroughly understanding integrable models that allow for exact algebraic computations and then appealing to universality phenomena that relate general models to integrable ones is common in probability theory.
For example, one may study general random walks on $\RR$ through the simple random walk on $\ZZ$ \cite{Donsker-1951-InvarianceLimitTheorems}, general random matrices through Gaussian random matrices like the Gaussian orthogonal ensemble \cite{ESY-2011-UniversalityRelaxationFlow}, and general surface growth models through ones with special combinatorial structure like the asymmetric simple exclusion process \cite{Corwin-2012-KPZEquationUniversality}.
Our aim here is to initiate such a program for low degree algorithms.
It is an interesting and important question to understand universality for other classes of algorithms and computations as well, which has in various senses been considered by recent works like \cite{DT-2019-UniversalityNumericalComputation,WZF-2022-UniversalityAMP,DLS-2023-UniversalityAMPSemirandom}.

\paragraph{Further integrable channels}
Similar calculations to those underlying Proposition~\ref{prop:gaussian-additive} are also possible for Bernoulli channels, using the Boolean Fourier basis, the orthogonal polynomials for $\Ber(\frac{1}{2})$, to give a closed form for the advantage.
The author's work \cite{Kunisky-2020-LowDegreeMorris} extended this kind of calculation to other channels where $\sP_x$ belong to a special class of exponential families, such as binomial, geometric, Poisson, and exponential, again using algebraic properties of the orthogonal polynomials for those families.
That work (as well as Appendix~B of \cite{BBKMW-2020-SpectralPlantingColoring} for the Bernoulli case) also gave comparison inequalities relating the advantage under such channels to the Gaussian advantage.
Our results may be seen as a major generalization of those, which still rely heavily on algebraic structure and thus only apply to a few integrable channels.

\paragraph{Channel universality}
Universality for statistical tasks with respect to the output channel has also been studied before.
The special case of spiked matrix models was studied by \cite{LKZ-2015-LowRankChannelUniversality,KXZ-2016-MutualInformationChannelUniversality,PWBM-2018-PCAI}, who noticed the central role of the Fisher information.
However, these results considered either statistical analysis or special algorithms like computing matrix eigenvalues and approximate message passing.
Ours is the first result in this direction to show channel universality of the behavior of a broad class of algorithms; perhaps our main conceptual observation is that the analysis of low degree algorithms is formally similar enough to a ``truncation'' of the statistical analysis that some key technical ideas of the above works still apply.
Taking a different approach, \cite{BBH-2019-UniversalityComputationalSubmatrixDetection} argued for channel universality for the particular problem of submatrix detection by considering reductions between different output channels.
This gives results applying to \emph{all} polynomial time algorithms, but is restricted to one specific class of priors $\sX$ specifying the submatrix detection problem, while our results allow for a broad range of $\sX$.

\paragraph{Low coordinate degree algorithms}
Since their proposal in Hopkins' thesis \cite{Hopkins-2018-Thesis}, LCDF have made a few further appearances.
In \cite{KM-2021-ReconstructionTreesLowDegree, HM-2024-LowDegreeBroadcastingTrees}, they were used to formulate LDP when the output of a channel is categorical, taking discrete values in some finite set of labels, $y_i \in \{\ell_1, \dots, \ell_k\}$.
In this case, LCDF and LDP are actually equivalent, so long as LDP are taken over a ``one-hot'' encoding of the $y_i$ as $(\One\{y_i = \ell_j\})_{i \in [N], j \in [k]}$ (see our Example~\ref{ex:boolean-fourier} and discussion throughout \cite{KM-2021-ReconstructionTreesLowDegree}).
In a different vein, \cite{BBHLS-2020-SQLowDegree} drew a connection between LCDF and the \emph{statistical query} model of hypothesis testing, where an algorithm is allowed a budget of arbitrary queries of a probability measure.
This work considered LCDF with a very ``coarse'' notion of coordinate, where one observes, e.g., many weakly informative samples of a random vector, each of which counts as one coordinate of the entire observation.
This is a very special case of the general latent variable model (without the ``continuity'' assumption) we give in Definition~\ref{def:lvm} where $\PP$ and $\QQ$ are both product measures, albeit over a higher-dimensional domain.

\paragraph{Efron-Stein decomposition}
In statistics and analysis, the idea of our technique for projecting the likelihood ratio to LCDF has variously been referred to as \emph{Efron-Stein, analysis of variance~(ANOVA),} or \emph{Hoeffding decomposition}.
It is presented, for instance, in Section~8.3 of \cite{ODonnell-2014-AnalysisBooleanFunctions}.
Its use in statistical applications is that a refinement of the subspaces $V_T$ that we present in Appendix~\ref{sec:coord-general} may be used to compute a measurement of the ``collective contribution from the coordinate set $T$'' to the total variance of the random variable $f(\by)$ for $\by \sim \QQ$, which underpins the ANOVA methodology.
In that language, our approach may be described as performing an ANOVA analysis of the likelihood ratio itself, studying how much of its variance is ``due to'' interactions among small numbers of coordinates of our observations.
The idea of projecting complicated functions to subspaces of functions of low coordinate degree has also been systematically explored for applications in physical sciences, especially in computational chemistry, under the name of \emph{high-dimensional model representation~(HDMR)}.
See \cite{LRR-2001-HDMR} for a general survey and \cite{RA-1999-FoundationsHDMR} for a more mathematical survey.

\paragraph{Channel calculus}
The effect of composing noisy channels has been much studied in information theory, though the Fisher information is a less common quantity in that literature.
Still, as for other more common measurements of a channel's ``fidelity'' like the mutual information, the Fisher information is known to satisfy a data processing inequality (so that, e.g., censorship or quantization cannot increase the Fisher information) and a chain rule \cite{Zamir-1998-FisherInformationDataProcessing}.

\section{Notation}

We write $[N] \colonequals \{1, \dots, N\}$.
The asymptotic notations $o(\cdot), O(\cdot), \omega(\cdot), \Omega(\cdot), \asymp$ have their usual meanings in the limit $N \to \infty$ or $n \to \infty$ when the variable $n$ is defined in context.
We write $\one_k \in \RR^k$ for the vector all of whose entries equal 1 and $\bm I_k \in \RR^{k \times k}$ for the $k \times k$ identity matrix.
For $\by$ an $N$-dimensional vector and $T \subseteq [N]$, we write $\by_T$ for the restriction to the index set $T$.
We write $\sN(\mu, \sigma^2)$ and $\sN(\bm\mu, \bm\Sigma)$ for the scalar and vector Gaussian measures with given mean and variance (respectively, covariance matrix) parameters.
We write $\delta_x$ for the Dirac probability measure at $x$, and $\Ber(c) = (1 - c)\delta_0 + c\delta_1$ for the Bernoulli measure.
This notation for mixtures will not be confused with $c \cdot \sX$, which denotes the measure formed by sampling from $\sX$ and then multiplying by $c$.
We adopt the slightly informal notation of denoting a channel $\sP = (\sP_x)$ by one of the above notations where $x$ is viewed as a ``dummy'' variable, so that, e.g., $\Adv_{\leq D}(\sX, \sN(x, \sigma^2)) = \Adv_{\leq D}(\sX, \sP)$ where $\sP_x = \sN(x, \sigma^2)$.

\section{Main results}
\label{sec:results}

\subsection{General theory}

Our first main result computes the coordinate advantage in closed form and gives a ``nonlinear overlap'' bound on it for any CLVM.\footnote{This may be viewed as a generalization of the bound derived in Appendix~B.1 of \cite{BBKMW-2020-SpectralPlantingColoring} of the Bernoulli advantage by the Gaussian advantage.}

\begin{definition}[Good CLVM]
    We call a CLVM \emph{good} if $\Sigma$ contains an open interval around zero, and, for all $x \in \Sigma$, $\sP_{x}$ is absolutely continuous with respect to $\sP_0$ and $d\sP_{x} / d\sP_0 \in L^2(\sP_0)$.
\end{definition}

\begin{definition}[Channel overlap]
    For each $x^{(1)}, x^{(2)} \in \Sigma$ in a good CLVM, define
    \begin{equation}
        R_{\sP}(x^{(1)}, x^{(2)}) \colonequals \Ex_{y \sim \sP_0}\left[\left(\frac{d\sP_{x^{(1)}}}{d\sP_0}(y) - 1\right)\left(\frac{d\sP_{x^{(2)}}}{d\sP_0}(y) - 1\right)\right].
    \end{equation}
\end{definition}

\noindent
The following examples are some simple cases of this function.
Note that near the origin each behaves, up to rescaling, like $x^{(1)}x^{(2)}$, which is the key phenomenon we will exploit later.

\begin{example}
    \label{ex:R-gauss}
    If $\sP_x = \sN(x, 1)$, then $R_{\sP}(x^{(1)}, x^{(2)}) = \exp(x^{(1)}x^{(2)}) - 1$.
\end{example}

\begin{example}
    \label{ex:R-ber}
    If $\sP_x = \Ber(\frac{1}{2} + x)$, then $R_{\sP}(x^{(1)}, x^{(2)}) = 4x^{(1)}x^{(2)}$.
\end{example}
\noindent
One should think of $R_{\sP}$ as a kind of ``kernel'' associated to a channel that contains all of the data about the channel that will be relevant to us.
In probabilistic terms, $R_{\sP}(x^{(1)}, x^{(2)})$ is the covariance between the likelihood ratios of $\sP_{x^{(1)}}$ and $\sP_{x^{(2)}}$ with respect to $\sP_0$, and thus is a measurement of the similarity between $\sP_{x^{(1)}}$ and $\sP_{x^{(2)}}$.
The diagonal values give the $\chi^2$ divergence (Definition~\ref{def:chi-squared}) between $\sP_x$ and $\sP_0$, $R_{\sP}(x, x) = \chi^2(\sP_x \dbar \sP_0)$.

\begin{theorem}[Coordinate advantage]
    \label{thm:lvm}
    Suppose $(\sX, \sP)$ is a good CLVM.
    Then,
    \begin{align}
      \CAdv_{\leq D}(\sX, \sP)^2
      &= \Ex_{\bx^{(1)}, \bx^{(2)} \sim \sX} \sum_{\substack{T \subseteq [N] \\ |T| \leq D}} \prod_{i \in T} R_{\sP}(x_i^{(1)}, x_i^{(2)}) \\
      &\leq \Ex_{\bx^{(1)}, \bx^{(2)} \sim \sX} \exp^{\leq D}\left(\sum_{i = 1}^NR_{\sP}(x_i^{(1)}, x_i^{(2)}) \right). \label{eq:lvm-lcdlr-exp}
    \end{align}
\end{theorem}
\noindent
We can actually treat a more general class of models than CLVMs as defined here, which we leave to Theorem~\ref{thm:lvm-general}.

The resemblance to Proposition~\ref{prop:gaussian-additive} is clear, with the two differences that $x_i^{(1)}x_i^{(2)}$ is replaced by the nonlinear $R_{\sP}(x_i^{(1)}, x_i^{(2)})$, and that the expression in \eqref{eq:lvm-lcdlr-exp} is only a \emph{bound} on the coordinate advantage rather than a \emph{formula} for it.
But, we will show that, in many cases, neither of these differences is very consequential, and for many channels the coordinate advantage actually behaves just like the polynomial advantage for a Gaussian channel.

We first address the matter of the nonlinear overlap.
In fact, we show that the summation in \eqref{eq:lvm-lcdlr-exp} in many cases may be replaced by an overlap formula of precisely the form that arose for the polynomial advantage for additive Gaussian models in Proposition~\ref{prop:gaussian-additive}.
The channel then enters into the bound not through the potentially complicated nonlinear $R_{\sP}$, but through the scalar \emph{Fisher information}, which describes $R_{\sP}$ near the origin (it is just the constant scaling $x^{(1)}x^{(2)}$ in the approximation alluded to for Examples~\ref{ex:R-gauss} and \ref{ex:R-ber}).
The reciprocal of the Fisher information plays the role of an ``effective $\sigma^2$'' in the formula for the additive Gaussian advantage.
Let us give a clearer name to the expression that will appear, to avoid confusion from mixing polynomial and coordinate advantages:
\begin{equation}
    \Univ_{\leq D}(\sX, \sigma^2) \colonequals \Ex_{\bx^{(1)}, \bx^{(2)}} \exp^{\leq D}\left(\frac{1}{\sigma^2}\langle \bx^{(1)}, \bx^{(2)}\rangle\right) = \Adv_{\leq D}(\sX, \sN(x, \sigma^2))^2.
\end{equation}
We also delineate the assumptions on the prior and channel that will play important roles.
When we refer to the constants $A, B_{k}$ later, they will always refer to the constants in these assumptions.
\begin{assumption}[Prior assumptions]
    \label{ass:prior}
    We define conditions on a prior $\sX$, parametrized by constants $A, B_2, B_4, B_6, B_8, B_{10}, B_{12} > 0$:
    \begin{enumerate}
    \item[P1.] $\|\bx\|_{\infty} \leq A$.
    \item[P2.] For each $k \in \{2, 4, 6\}$, $\|\bx\|_k \leq B_kN^{\frac{1}{k} - \frac{1}{4}}$.
    \item[P3.] For each $k \in \{8, 10, 12\}$, $\|\bx\|_k \leq B_kN^{\frac{1}{k} - \frac{1}{4}}$.
    \end{enumerate}
\end{assumption}

\begin{assumption}[Channel assumptions]
    \label{ass:channel}
    We define conditions on a channel $\sP$, parametrized by a constant $A > 0$ (always the same as $A$ in Assumption~\ref{ass:prior}):
    \begin{enumerate}
    \item[C1.] $R_{\sP}(x^{(1)}, x^{(2)})$ is $\sC^4$ in an open set containing $[-A, A]^2$.
    \item[C2.] $\frac{\partial^3R_{\sP}}{\partial x^{(1)^2} \partial x^{(2)}}(0, 0) = 0$.
    \end{enumerate}
\end{assumption}

\begin{definition}[Fisher information]
    For a channel $\sP$ satisfying Assumption C1, its \emph{Fisher information} is
    \begin{equation}
        F_{\sP} \colonequals \frac{\partial^2R_{\sP}}{\partial x^{(1)} \partial x^{(2)}}(0, 0).
    \end{equation}
\end{definition}
\noindent
The Fisher information in statistics is usually viewed as a function of the signal $x$, $F_{\sP}(x)$, and in this language what we work with is $F_{\sP} = F_{\sP}(0)$ (see also Proposition~\ref{prop:fisher-info-equiv}).
Its values for the Gaussian and Bernoulli channels from Examples~\ref{ex:R-gauss} and \ref{ex:R-ber} are 1 and 4, respectively.

\begin{theorem}[Loose channel universality]
    \label{thm:channel-universality}
    Let $(\sX, \sP)$ be a good CLVM.
    Suppose that $\sP$ satisfies Assumptions~C1 and C2, and $\sX$ satisfies Assumptions~P1 and P2.
    Then, there is a constant $C_1 > 0$ depending only on the channel $\sP$ and the constants $(A, B_2, B_4, B_6)$ in the Assumptions such that, for all $D \geq 0$ even,
    \begin{equation}
        \CAdv_{\leq D}(\sX, \sP)^2 \leq C_1\,\Univ_{\leq D}(\sX, 1 / F_{\sP}). \label{eq:cadv-upper}
    \end{equation}
    Suppose further that $\sX$ satisfies Assumption~P3.
    Then, there are also constants $C_2, C_3 > 0$ depending only on the channel $\sP$ and on the constants $(A, B_2, B_4, B_6, B_8, B_{10}, B_{12})$ such that, for all $D \geq 0$ even,
    \begin{equation}
      \CAdv_{\leq D}(\sX, \sP)^2 \geq C_2\,\Univ_{\leq D}(\sX, 1 / F_{\sP}) - C_3\,\Univ_{\leq D - 2}(\sX, 1 / F_{\sP}). \label{eq:cadv-lower}
    \end{equation}
\end{theorem}
\noindent
To understand the result, first recall that, for studying hardness of detection, we are interested in understanding boundedness or divergence of $\CAdv_{\leq D}$ over a sequence of CLVMs.
The upper and lower bounds in \eqref{eq:cadv-upper} and \eqref{eq:cadv-lower} will match up to constants for $D = \omega(1)$ so long as $\Univ_{\leq D}(\sX, 1 / F_{\sP})$ either is bounded (so that the upper and lower bounds are both constant) or grows at a sufficiently fast exponential rate in $D$ (so that the second term in \eqref{eq:cadv-lower} becomes negligible and we find $\CAdv_{\leq D}(\sX, \sP)^2 \asymp \Univ_{\leq D}(\sX, 1 / F_{\sP})$).
Because of this ``sufficiently fast'' clause, this result gives universality of detection thresholds only up to constant factors in an SNR parameter (see Remark~\ref{rem:snr}; roughly speaking, the actual exponential growth of $\Univ_{\leq D}(\sX, 1 / F_{\sP})$ is usually at a rate proportional to how much greater the SNR is than the threshold for detection), which is why we call the result ``loose.''
For a special class of priors $\sX$, we will later give a tight characterization of the coordinate advantage up to constants, thus giving true universality of computational thresholds; however, this more \emph{ad hoc} bound is useful to treat arbitrary priors.

Let us parse the conditions of the Theorem and give some remarks on its applicability.
Assumptions P1--P3 on the prior may be viewed as asking that the typical value of $|x_i|$ is at most $N^{-1/4}$, allowing for a small number of outliers.
This is a natural scaling, occurring for instance in low-rank spiked matrix models (see Corollary~\ref{cor:non-gaussian-smm} and Section~\ref{sec:spiked-mx}).
In Example~\ref{ex:non-universality} we give an illustration that some assumption to this effect is necessary for universality to hold.

\begin{remark}[Truncating priors]
    \label{rem:trunc-prior}
    By a common trick, used for example in \cite{PWBM-2018-PCAI,BKW-2019-ConstrainedPCA}, it suffices, if we are assuming Conjecture~\ref{conj:low-deg} and reasoning about a sequence of problems $(\sX_n, \sP)$ over a fixed channel, to have the Assumptions~P1--P3 hold with high probability as $n \to \infty$.
    This is because we may define an adjusted prior $\widetilde{\sX}_n$ where we sample $\widetilde{\bx} \sim \widetilde{\sX}_n$ as $\widetilde{\bx} = \bx \One\{E\}$ for $\bx \sim \sX$ and $E$ the event that P1 and P2 hold.
    Then, if we prove that strong detection is hard for $\widetilde{\sX}_n$, we automatically learn the same for $\sX_n$, since the two agree with high probability.
\end{remark}

The conditions C1 and C2 on the channel are harder to parse, but let us demonstrate how mild they are.
The following two results show that nearly arbitrary additive noise channels and nearly arbitrary channels where observations are made in an exponential family satisfy these conditions.
The result for additive noise is similar in spirit to results of \cite{PWBM-2018-PCAI}, and the one for exponential families to those of \cite{Kunisky-2020-LowDegreeMorris}, though in both cases our results have more general consequences.
We give the proofs---straightforward verifications of the Assumptions---in Appendices~\ref{sec:pf:prop:channel-additive} and~\ref{sec:pf:prop:channel-exponential}.

\begin{proposition}[Additive noise channels]
    \label{prop:channel-additive}
    Suppose that $\sP$ is a channel such that, for all $x \in \Sigma$, $\sP_x$ is the law of $x + z$ for $z \sim \rho$ with $\rho$ a probability measure on $\RR$ having a density $p(y)$ with respect to Lebesgue measure that satisfies:
    \begin{enumerate}
    \item $p(y) > 0$ for all $y \in \RR$.
    \item $p(y)$ is $\sC^4$ on all of $\RR$.
    \item $p(y) = p(-y)$ for all $y \in \RR$.
    \end{enumerate}
    Then, Assumptions C1 (for arbitrary choice of the constant $A$) and C2 are satisfied.
    The Fisher information is given in terms of $p$ by
    \begin{equation}
        F_{\sP} = \int_{-\infty}^{\infty} \frac{p^{\prime}(y)^2}{p(y)} dy.
    \end{equation}
\end{proposition}

\begin{proposition}[Exponential family channels]
    \label{prop:channel-exponential}
    Suppose that $\sP_0$ has mean $\mu$ and variance $\sigma^2 > 0$, and, for all $x \in \Sigma$, $\sP_x$ belongs to the natural exponential family generated by $\sP_0$ (see Definition~\ref{def:exp-family}) and has $x = \EE_{y \sim \sP_x}[y] - \mu$.
    Suppose also that $\sP_0$ has a cumulant generating function $\psi(\theta) \colonequals \log \EE_{y \sim \sP_0} \exp(\theta y)$ that satisfies:
    \begin{enumerate}
    \item $\psi$ is $\sC^5$ in an open set $U$ containing zero.
    \item $\psi^{\prime}(U)$ contains an open set $U^{\prime} \supset [-A + \mu, A + \mu]$.
    \end{enumerate}
    Then, Assumptions C1 and C2 are satisfied, where $A$ is the same as the constant in Assumption~C1.
    The Fisher information is
    \begin{equation}
        F_{\sP} = \frac{1}{\sigma^2}.
    \end{equation}
\end{proposition}

\begin{example}
    If $\EE_{y \sim \sP_0}\exp(\theta y) < \infty$ for all $\theta \in \RR$, e.g.\ if $\sP_0$ is subgaussian (with any variance proxy), then $\psi$ is automatically smooth on all of $\RR$ and all three conditions are satisfied.
\end{example}

\begin{remark}[Difference-of-means parametrization]
    The condition $x = \EE_{y \sim \sP_x}[y] - \mu$ says that not only do the $\sP_x$ belong to an exponential family, but also that they are specifically parametrized by the displacement of their means from that of $\sP_0$.
    This is important to our calculations in Section~\ref{sec:pf:prop:channel-exponential} and to an ``overlap formula'' holding, and is not original; the idea that the ``right'' overlap controlling exponential family channels is an overlap of $z$-scores (as we obtain if we move the factor of $F_{\sP} = 1/\sigma^2$ inside of $\langle \bx^1, \bx^2 \rangle$ in Theorem~\ref{thm:channel-universality}) appeared already in \cite{Kunisky-2020-LowDegreeMorris}.
\end{remark}

\begin{remark}[Exponential families are hardest]
    \label{rem:exp-hardest}
    In \cite{PWBM-2018-PCAI}, it is noted that, among additive noise models as in Proposition~\ref{prop:channel-additive}, for a given variance the Gaussian distribution uniquely minimizes the Fisher information: for $\sP_x = \sN(x, \sigma^2)$, we have $F_{\sP} = 1 / \sigma^2$, which per Corollary~\ref{cor:fisher-info-lb} is the smallest possible value.
    The authors conclude (for their particular spiked matrix problem) that Gaussian is therefore the ``hardest'' distribution of additive noise.
    Our Proposition~\ref{prop:channel-exponential} shows that, if we move beyond additive noise, all exponential family channels are ``as hard as'' the additive Gaussian channel with the same variance.
    This perhaps clarifies the ``Gaussian is hardest'' phenomenon: generally it is exponential family channels that are hardest; the key property of Gaussian measure is that the Gaussian additive channel is the only additive noise channel whose measures $\sP_x$ also form an exponential family.
\end{remark}

In our final general result, we consider when the upper bound of Theorem~\ref{thm:channel-universality} is tight up to constants, in which case whether $\CAdv_{\leq D}(\sX_n, \sP)$ is bounded or divergent as $n \to \infty$ is truly universal with respect to the channel.
\begin{definition}[Prior dilution]
    Let $\sX$ be a probability measure over $\Sigma^N$ as in the definition of a CLVM, and let $k \geq 1$.
    The \emph{$k$-dilution} of $\sX$, denoted $\sD_k \sX$, is the probability measure over $\Sigma^{kN}$ where, to sample $\widetilde{\bx} \sim \sD_k \sX$, we sample $\bx \sim \sX$ and set $\widetilde{\bx} = \frac{1}{\sqrt{k}}(x_1, \dots, x_1, \dots, x_N, \dots, x_N)$, where each entry is repeated $k$ times (equivalently, $\widetilde{\bx} = \bx \otimes \frac{1}{\sqrt{k}}\one_k$).
\end{definition}

\begin{theorem}[Tight channel universality]
    \label{thm:dilution}
    Let $(\sX, \sP)$ be a good CLVM.
    Suppose that $\sX$ satisfies Assumptions P1 and P2, that $\sP$ satisfies Assumptions C1 and C2, and that $k \geq D^2$.
    Then, there are constants $C_1, C_2$ only depending on the channel $\sP$ and on the constants $(A, B_2, B_4, B_6)$ such that
    \begin{equation}
        C_1\,\Univ_{\leq D}(\sX, 1 / F_{\sP}) \leq \CAdv_{\leq D}(\sD_k \sX, \sP)^2 \leq C_2\,\Univ_{\leq D}(\sX, 1 / F_{\sP}).
    \end{equation}
\end{theorem}
\noindent
We note that dilution, in addition to giving this tight characterization of the coordinate advantage, can also cause a prior to satisfy Assumptions P1--P3, since $\|\bx\|_2$ is unchanged by dilution, while $\|\bx\|_{2k}$ for $k \geq 2$ only decrease.

The idea of diluting priors for the analysis of low degree algorithms (though for a quite different purpose of relating low degree algorithms to the statistical query model) appeared previously in \cite{BBHLS-2020-SQLowDegree}.
As discussed there, for the Gaussian channel, an algorithm may itself generate more dilute samples from given less dilute ones.
\begin{proposition}[Gaussian dilution invariance; Lemma 7.2 of \cite{BBHLS-2020-SQLowDegree}]
    \label{prop:gaussian-dilution}
    There are randomized polynomial-time algorithms computing $f: \RR \to \RR^k$ and $g: \RR^k \to \RR$ such that:
    \begin{itemize}
    \item If $y \sim \sN(\mu, \sigma^2)$, then $f(y) \sim \sN(\mu \bm 1_k / \sqrt{k}, \sigma^2 \bm I_k)$.
    \item If $\by \sim \sN(\mu \bm 1_k / \sqrt{k}, \sigma^2 \bm I_k)$, then $g(\by) \sim \sN(\mu, \sigma^2)$.
    \end{itemize}
\end{proposition}
\noindent
In this sense, Gaussian models are a ``fixed point of dilution,'' which may be seen as a justification for why sufficiently dilute models behave like Gaussian ones.
More specifically, as we sketch in Section~\ref{sec:pf:thm:dilution}, in a diluted model, an algorithm may compute an approximately \emph{minimum variance unbiased estimator} of $x_i$ from the $k$ samples from $\sP_{x_i / \sqrt{k}}$ and average them to obtain an observation of $x_i$ that will, by the central limit theorem, be asymptotically Gaussian.
A version of the Cram\'{e}r-Rao lower bound implies that the lowest variance achievable by such an estimate is precisely $1 / F_{\sP}$.
If this holds, then the resulting averages will be (approximately) distributed as $\sN(x_i, 1 / F_{\sP})$, thus reducing the dilute model with an arbitrary channel to a corresponding less dilute Gaussian model of variance $1 / F_{\sP}$.

\begin{remark}[Channel universality of $\chi^2$ divergence]
    As an aside, we note that if we take $D = N$ in the coordinate advantage, we obtain the $L^2(\QQ)$ norm of the likelihood ratio, which is $1 + \chi^2(\PP \dbar \QQ)$, where the latter is the \emph{$\chi^2$ divergence}.
    Thus, at $D = N$, our results give channel universality for the $\chi^2$ divergence, an interesting and new phenomenon in itself, much in the spirit of the channel universality of the mutual information studied by \cite{LKZ-2015-LowRankChannelUniversality,KXZ-2016-MutualInformationChannelUniversality} for spiked matrix models.
    The $\chi^2$ divergence underlies the \emph{second moment method} for establishing statistical lower bounds for hypothesis testing \cite{MRZ-2015-LimitationsSpectral, BMVVX-2018-InfoTheoretic, PWBM-2018-PCAI}.
    However, it is often augmented with \emph{conditioning} rather than merely computing and bounding the divergence itself; it is likely possible to use our results in tandem with these techniques, but we leave this investigation to future work.
    See also Appendix~\ref{sec:coord-remarks} for more connections between our calculations and the $\chi^2$ divergence.
\end{remark}

\subsection{Applications}
\label{sec:app}

We now move to more concrete considerations, and give two applications of our framework to identifying computationally hard regimes.
In both applications, we will work with priors built from the following kind of random vectors.
\begin{assumption}[Spike priors]
    \label{ass:spike-prior}
    Suppose that $\pi$ is a bounded probability measure on $\RR$ with $\EE_{x \sim \pi}[x] = 0$ and $\EE_{x \sim \pi}[x^2] = 1$.
    We will consider $\bx \in \RR^n$ random with $x_i \sim \pi$ i.i.d.
\end{assumption}
\noindent
We could somewhat relax the boundedness condition and obtain similar results by using the truncation idea in Remark~\ref{rem:trunc-prior}; see, e.g., the general notion of ``tame priors'' from \cite{Kunisky-2021-SpectralBarriersCertification}.

\paragraph{Spiked matrix models}
The first application revisits the non-Gaussian spiked matrix models studied by \cite{LKZ-2015-LowRankChannelUniversality, PWBM-2018-PCAI}.
Let $\sP$ be an additive noise channel satisfying the assumptions of Proposition~\ref{prop:channel-additive}, with underlying density $p$.
Let $\lambda > 0$.
For $\bx$ as in Assumption~\ref{ass:spike-prior}, let $\sX_n$ denote the law of the upper triangle of $\frac{\lambda}{\sqrt{n}}\bx\bx^{\top}$ (not including the diagonal).
We view draws from $\sX_n$ and associated CLVMs as symmetric matrices whose diagonal is zero, by repeating every entry from the upper triangle in the lower triangle.

The following result analyzes a natural algorithm for testing in such a model by thresholding the largest eigenvalue of the observed matrix after an entrywise transformation.
One of the main insights of \cite{LKZ-2015-LowRankChannelUniversality, PWBM-2018-PCAI} is that this transformation is necessary to obtain an algorithm that performs optimally.

\begin{proposition}[Pretransformed eigenvalue test; Theorem 4.8 of \cite{PWBM-2018-PCAI}]
    \label{prop:pwbm}
    Suppose that $\lambda > 1 / \sqrt{F_{\sP}}$, write  $f(y) \colonequals -p^{\prime}(y) / p(y)$, and make the following additional assumptions on the density $p$:
    \begin{enumerate}
    \item $f$ and its first two derivatives are polynomially bounded, i.e., $|f^{(\ell)}(y)| \leq C + y^m$ for some $C > 0$ and even $m \geq 2$ and each $\ell \in \{0, 1, 2\}$.
    \item $p$ has finite moments up to order $5m$: $\int_{-\infty}^{\infty} |x|^k < \infty$ for $1 \leq k \leq 5m$.
    \end{enumerate}
    Write $f(\bY)$ for the entrywise application of $f$ to a matrix $\bY$.
    Define the function
    \begin{equation}
        \mathsf{test}(\bY) \colonequals \left\{\begin{array}{ll} \texttt{q} & \text{if } n^{-1/2}\lambda_{\max}(f(\bY)) < \sqrt{F_{\sP}} + \frac{1}{2}\lambda F_{\sP} + \frac{1}{2\lambda}, \\ \texttt{p} & \text{if } n^{-1/2}\lambda_{\max}(f(\bY)) \geq \sqrt{F_{\sP}} + \frac{1}{2}\lambda F_{\sP} + \frac{1}{2\lambda} \end{array}\right.
    \end{equation}
    Then, $\mathsf{test}$ runs in time $\poly(n)$ and achieves strong detection in the CLVM $(\sX_n, \sP)$.
\end{proposition}

\noindent
Our results imply the following complementary result.

\begin{corollary}[LCDF analysis of spiked matrix model]
    \label{cor:non-gaussian-smm}
    The following hold:
    \begin{enumerate}
    \item If $\lambda < 1 / \sqrt{F_{\sP}}$, then for any $D = D(n) = o(n / \log n)$, $\CAdv_{\leq D}(\sX_n, \sP) = O(1)$. Consequently, if Conjecture~\ref{conj:low-deg} holds, then, for any $\delta > 0$, there is no algorithm that runs in time $\exp(O(n^{1 - \delta}))$ and achieves strong detection in the CLVM $(\sX_n, \sP)$.
    \item If $\lambda > 1 / \sqrt{F_{\sP}}$, then for any $D = D(n) = \omega(\log n)$, $\CAdv_{\leq D}(\sX_n, \sP) = \omega(1)$.
    \end{enumerate}
\end{corollary}
\noindent
Thus, to improve on the pretransformed eigenvalue test requires nearly exponential time.
The proof of Claim 1 is very simple, combining our Theorem~\ref{thm:channel-universality} with the previous results of \cite{KWB-2022-LowDegreeNotes} on additive Gaussian models.
Claim 2 encodes a version of Proposition~\ref{prop:pwbm} by using the trace of a large power of a matrix as a proxy for the spectral norm.
We emphasize that, despite this simplicity, Claim 1 above gives the first evidence of hardness for any large class of algorithms for all but a few very special cases of the density $p$.

\begin{remark}[Higher rank priors]
    \label{rem:higher-rank}
    The prior work \cite{BBKMW-2020-SpectralPlantingColoring} established limitations of LDP for detecting matrices of constant rank $k \geq 2$ as $n \to \infty$ observed through additive Gaussian noise.
    We omit the details, but the proof of Corollary~\ref{cor:non-gaussian-smm} \emph{mutatis mutandis} also extends those limitations to LCDF and establishes their channel universality, so long as $\lambda$ is again scaled by $\sqrt{F_{\sP}}$.
\end{remark}

One interesting example is the \emph{sparse Rademacher} prior where, for some $s \in (0, 1]$, $\bx$ in the prior has entries distributed as
\begin{equation}
    \pi = (1 - s)\delta_0 + \frac{s}{2}\delta_{1 / \sqrt{s}} + \frac{s}{2}\delta_{-1 / \sqrt{s}}.
\end{equation}
As predicted using non-rigorous statistical physics methods by \cite{LKZ-2015-PhaseTransitionsSparsePCA,LKZ-2015-LowRankChannelUniversality} and proved by \cite{KXZ-2016-MutualInformationChannelUniversality}, there is some $s^* \approx \num{0.09}$ such that, once $s < s^*$, there is an exponential-time exhaustive search algorithm achieving strong detection in the CLVM $(\sX_n, \sP)$ for some values of $\lambda < 1 / \sqrt{F_{\sP}}$.\footnote{This is not entirely explicit in those results, but, for discrete priors such as sparse Rademacher, one may view their results as computing explicitly the typical value of the \emph{free energy}, a particular function of the observation in a spiked matrix model which essentially coincides with the likelihood ratio. Their results then imply that thresholding the free energy achieves strong detection, which amounts to analyzing the statistically optimal hypothesis test per the Neyman-Pearson lemma.}
In this context, our Corollary~\ref{cor:non-gaussian-smm} gives evidence that this problem has a statistical-to-computational gap for \emph{any} reasonable choice of additive noise.
Again, this is the first evidence---in the form of lower bounds against any large class of algorithms---for this statistical-to-computational gap for most noise distributions.

\begin{remark}[Stochastic block model]
    \label{rem:sbm}
    Considering the sparse Rademacher prior through the Bernoulli channel $\sP_x = \Ber(\frac{1}{2} + x)$, we obtain a model that is a dense version of the \emph{stochastic block model} (see \cite{Abbe-2017-SBMReview, Moore-2017-SBMReview} for general background): we seek to distinguish an \Erdos-\Renyi\ random graph on $n$ vertices with edge probability $\frac{1}{2}$ from a graph with two random ``planted communities,'' each of roughly $\frac{s}{2}\cdot n$ vertices, such that vertices are connected with probability $\frac{1}{2} + \frac{\lambda}{sn}$ within those communities, probability $\frac{1}{2} - \frac{\lambda}{sn}$ between them, and probability $\frac{1}{2}$ otherwise.
    The same argument as for the first claim of Corollary~\ref{cor:non-gaussian-smm} shows that LCDF fail to achieve strong detection in this model when $\lambda < 1 / \sqrt{F_{\sP}} = 1/2$, and the results of \cite{KXZ-2016-MutualInformationChannelUniversality} imply a statistical-to-computational gap for sufficiently small $s$.
    Similar results also follow for observations in other non-Bernoulli discrete exponential families, such as Poisson, binomial, or geometric.
    We do not pursue it here, but, following Remark~\ref{rem:higher-rank}, greater numbers of communities may also be treated by considering a signal matrix of higher rank, as in \cite{BBKMW-2020-SpectralPlantingColoring}.
\end{remark}

\paragraph{Spiked tensor models}
The second application concerns the problem of \emph{tensor PCA}, for which, to the best of our knowledge, general non-Gaussian additive noise has not been considered before.
Let $q \geq 3$.
We will now allow $\lambda = \lambda(n) > 0$ to vary.
For $\bx$ as in Assumption~\ref{ass:spike-prior}, let $\sX_n$ denote the law of the entries of $\lambda n^{-q/4} \bx^{\otimes q}$ indexed by tuples $1 \leq i_1 < \cdots < i_q \leq N$.
(Again, one may view draws from the prior and observations from the CLVM as being symmetric tensors with all entries having repeated indices set to zero.)

The following result is a sharpening of prior work of \cite{HKPRSS-2017-SOSSpectral, Hopkins-2018-Thesis} that precisely characterizes (in the low degree framework) the power of subexpoential time algorithms for tensor PCA.

\begin{proposition}[Theorem 3.3 of \cite{KWB-2022-LowDegreeNotes}; Theorem 5.2.6 of \cite{Kunisky-2021-SpectralBarriersCertification}]
    \label{prop:kwb-stm}
    Let $\sP_x = \sN(x, 1)$.
    Then, there are constants $a_q, b_q > 0$ such that:
    \begin{enumerate}
    \item If $\lambda \leq a_q D^{-(q - 2)/4}$, then $\Adv_{\leq D}(\sX_n, \sP) = O(1)$.
    \item If $\lambda \geq b_q D^{-(q - 2)/4}$, $D = \omega(1)$, and $D \leq \frac{2}{q}n$, then $\Adv_{\leq D}(\sX_n, \sP) = \omega(1)$.
    \end{enumerate}
\end{proposition}

As this result shows, unlike spiked matrix models, spiked tensor models have a smoother statistical-to-computational ``ramp'' with a large parameter regime where subexponential time algorithms exist, which manifests as a polynomial dependence on $D$ in the thresholds for $\lambda$.
Because of this, it is less meaningful here to pin down the constants $a_q$ and $b_q$.
Allowing this slack makes our machinery even more useful, and we may with a very straightforward application of Theorem~\ref{thm:channel-universality} obtain the following direct generalization to arbitrary additive noise models.

\begin{corollary}[LCDF analysis of spiked tensor model]
    \label{cor:non-gaussian-stm}
    Let $\sP$ be an additive noise channel satisfying the assumptions of Proposition~\ref{prop:channel-additive}.
    Then, there are constants $a_{q, \sP}, b_{q, \sP} > 0$ such that:
    \begin{enumerate}
    \item If $\lambda \leq a_{q, \sP} D^{-(q - 2)/4}$, then $\CAdv_{\leq D}(\sX_n, \sP) = O(1)$.
    \item If $\lambda \geq b_{q, \sP} D^{-(q - 2)/4}$, $D = \omega(1)$, and $D \leq \frac{2}{q}n$, then $\CAdv_{\leq D}(\sX_n, \sP) = \omega(1)$.
    \end{enumerate}
\end{corollary}

Note that, unlike in the spiked matrix model where the pretransformed eigenvalue test had been studied before, here we are able to make predictions about computational thresholds for problems where explicit algorithms have not been studied at all (for general choices of the density $p$).
Our result suggests that there is \emph{some} LCDF depending on the density $p$ that achieves strong detection (say, by thresholding as in Proposition~\ref{prop:pwbm}) up to the value of $\lambda$ we predict, but we may generate our predictions without even writing this function down explicitly.

\subsection{Channel calculus}

Finally, we elaborate on our results on modifying noisy channels.

\begin{definition}[Censorship]
    Let $\sP$ be a channel with range $\Omega$ and $\eta \in [0, 1]$.
    We define the \emph{$\eta$-censored} version of $\sP$ to be the channel $\sC_{\eta}\sP$ with range $\Omega \sqcup \{\bullet\}$ for a new formal symbol $\bullet$, where to sample from the measure $\sC_{\eta}\sP_x$, we observe $\bullet$ with probability $\eta$, and a sample from $\sP_x$ with probability $1 - \eta$.
\end{definition}

\begin{theorem}[Censored channels]
    \label{thm:censorship}
    Suppose $\sP$ is a channel satisfying Assumptions C1 and C2.
    Then, for any $\eta \in [0, 1]$, $\sC_{\eta}\sP$ also satisfies Assumptions C1 and C2, and has Fisher information $F_{\sC_{\eta}\sP} = (1 - \eta) F_{\sP}$.
\end{theorem}
\noindent
The theorem is simple both to state and to prove, but is deceptively powerful.
For example, combined with Corollary~\ref{cor:non-gaussian-smm}, we deduce that strong detection in the spiked matrix model with additive noise channel $\sP$ censored at a constant rate $\eta$ is hard for LCDF when $\lambda < 1 / \sqrt{(1 - \eta)F_{\sP}}$ (for a censored spiked tensor model, Corollary~\ref{cor:non-gaussian-smm} holds verbatim with constants $a_{q, \sP, \eta}$ and $b_{q, \sP, \eta}$ adjusted to absorb the dependence on $\eta$).

Again, to the best of our knowledge these models have not been studied before, and our theory can make predictions of computational thresholds even in the absence of ``baseline'' algorithms to consider.
In Section~\ref{sec:censorship}, we give some speculation as to a spectral algorithm we expect to match the aforementioned threshold for detecting censored spiked matrices, and propose the related Conjecture~\ref{conj:censored-rmt}, a claim of a random matrix phase transition similar to those of \cite{BBAP-2005-LargestEigenvalueSampleCovariance, FP-2007-LargestEigenvalueWigner}.

\begin{remark}[Censored stochastic block model]
    \label{rem:censored-sbm}
    The previous work \cite{SLKZ-2015-SpectralDetectionCensoredSBM} considered a censored version of the dense stochastic block model we mentioned in Remark~\ref{rem:sbm}.
    They consider the dense Rademacher prior ($s = 1$) and a very high rate of censorship where $\eta = 1 - \frac{\alpha}{n}$ for some constant $\alpha > 0$, while the prior $\sX$ is the law of $\lambda\bx\bx^{\top}$ for $\lambda$ constant, without a factor of $\frac{1}{\sqrt{n}}$.
    But, one may check that the factors of $\alpha$ and $n$ may be moved around the components of the model: for the Bernoulli channel $\sP_x = \Ber(\frac{1}{2} + x)$, the coordinate advantage in their model is identical to that of a model with the Rademacher prior as we defined it above (with the factor of $\frac{1}{\sqrt{n}}$) and with $\lambda$ multiplied by $\sqrt{\alpha}$.
    Thus, we immediately recover that LCDF have the same computational threshold as proved there, which in our notation is $\lambda > \frac{1}{2\sqrt{\alpha}}$.
    (For comparison, \cite{SLKZ-2015-SpectralDetectionCensoredSBM} use a parameter $\epsilon$ which is $\epsilon = \frac{1}{2} - \lambda$, and write the threshold as $\alpha > \frac{1}{(1 - 2\epsilon)^2}$.)
    This problem is known, by their results together with \cite{HLM-2012-LabelledSBM,LMX-2015-LabelledSBM}, not to have a statistical-to-computational gap, but once again our previous reasoning gives that if we took a sparser community structure with $s < 1$, then for sufficiently small $s$ a gap would appear.
    And, once again, our tools immediately give analogous results for observations in other discrete exponential families.
\end{remark}

\begin{remark}[Effective signal-to-noise ratio]
    More generally, when a prior has an SNR $\lambda$ that appears linearly as above, for a given channel $\sP$ censored at rate $\eta$, we may define the \emph{effective SNR} $\lambda_{\eff} \colonequals \sqrt{(1 - \eta)F_{\sP}} \cdot \lambda$.
    The coordinate advantage of the model then behaves like that of an additive Gaussian model with SNR $\lambda_{\eff}$.
    It is an intriguing question what other aspects of a model can easily be ``factored in'' to an effective SNR in such generality.
\end{remark}

We now proceed to the second operation of \emph{quantization} of a channel's output.
We note that we consider a coarse notion of quantization into just one bit of information; it would also be interesting to consider quantizations on a finer grid (though the width of this grid would have to be $O(N^{-1/4})$ to obtain substantially different behavior from just taking the sign for additive $\sP$ and priors satisfying Assumptions P1--P2).

\begin{definition}[Quantization]
    Let $\sP$ be a channel with range $\Omega \subset \RR$ so that each $\sP_x$ is absolutely continuous with respect to Lebesgue measure.
    We define the \emph{quantized} version of $\sP$ to be the channel $\sgn(\sP)$ with range $\{-1, 1\}$ (the value zero may be mapped arbitrarily), where to sample from the measure $\sgn(\sP)_x$ we sample $y \sim \sP_x$ and output $\sgn(y)$.
\end{definition}

\begin{theorem}[Quantized additive channels]
    \label{thm:quantization}
    Let $\sP$ be an additive noise channel as in Proposition~\ref{prop:channel-additive}, with underlying density $p(y)$.
    If $p$ satisfies the conditions of the Proposition, then $\sgn(\sP)$ satisfies Assumptions C1 and C2, and has Fisher information $F_{\sgn(\sP)} = 4p(0)^2$.
\end{theorem}
\noindent
Again, the result follows by a simple calculation, but immediately pinpoints computational thresholds of quantized spiked matrix and tensor models in a simple way depending on the density of additive noise.

\begin{example}[Quantized additive Gaussian noise]
    \label{ex:quantized-additive-gaussian}
    Taking $\sP_x = \sN(x, 1)$, we find $F_{\sgn(\sP)} = 4(\frac{1}{\sqrt{2\pi}})^2 = \frac{2}{\pi}$, which is indeed smaller than $F_{\sP} = 1$.
    We find that, e.g., the computational threshold for LCDF for a spiked matrix model with the channel $\sgn(\sN(x, 1))$ is $\lambda > 1 / \sqrt{F_{\sgn(\sP)}} = \sqrt{\pi / 2}$.
\end{example}

\noindent
And, again, we do not know what algorithms would match this threshold, but we conjecture that spectral algorithms should do so and make the corresponding Conjecture~\ref{conj:signed-rmt}.

The expression $4p(0)^2$ is rather mysterious; it is not even obvious from it that quantizing decreases the Fisher information, though it must by the aforementioned data processing inequality of \cite{Zamir-1998-FisherInformationDataProcessing}.
Seeking to the maximize the quantized Fisher information for a given unquantized Fisher information, we find the following surprising result.
In words, it says that there are additive noise models with smooth noise distributions so that quantization has an arbitrarily small effect on computational thresholds for hypothesis testing!
\begin{corollary}
    \label{cor:exp-quantization}
    For every $\epsilon > 0$, there exists a smooth and strictly positive probability density $p_{\epsilon}$ such that, if $\sP$ is the additive noise channel with this density, then $(1 - \epsilon)F_{\sP} \leq F_{\sgn(\sP)} \leq F_{\sP}$.
    Moreover, the $p_{\epsilon}$ may be chosen to converge pointwise to another positive (but not smooth) probability density, $p_{\epsilon}(y) \to p_0(y) \colonequals \frac{1}{2}\exp(-|y|)$ as $\epsilon \to 0$.
\end{corollary}
\noindent
We give a partial account of this paradoxical result in Section~\ref{sec:quantization}.

\section{Latent variable models}

As mentioned earlier, we will work at first over a more general model than the CLVMs of Definition~\ref{def:clvm}.
The following definition is more general in two important ways, and a third less material one: first, it does not ask for the domain of $\sX$ to be continuous; second, it does not require $y_i$ to be independent of all $x_j$ for $j \neq i$; and third, it allows for the $y_i$ to have different laws for different $i$ under the null model $\QQ$ and to be observed through different channels under the planted model $\PP$.
We hope this more abstract formulation will be useful for future work.
\begin{definition}[Latent variable model]
    \label{def:lvm}
    Let $\Sigma, \Omega$ be measurable spaces and $N \geq 1$.
    Let $\sX$ be a probability measure over $\Sigma^N$, $\sQ_1, \dots, \sQ_N$ probability measures over $\Omega$, and, for each $\bx \in \Sigma$ and $i \in [N]$, let $\sP_{i, \bx}$ be a probability measure over $\Omega$.
    A \emph{latent variable model (LVM)} specified by these objects $(\sX, \sQ, \sP)$ consists of the following pair of probability measures over $\Omega^N$:
    \begin{enumerate}
    \item Sample $\by \sim \QQ$ by sampling $y_i \sim \sQ_i$ for each $i \in [N]$ independently (i.e., $\QQ$ is the product measure $\QQ = \sQ_1 \otimes \cdots \otimes \sQ_N$).
    \item Sample $\by \sim \PP$ by first sampling $\bx \sim \sX$. Then, sample $y_i \sim \sP_{i, \bx}$ for each $i \in [N]$ independently (i.e., $\PP$ is the mixture of product measures $\sP_{1, \bx} \otimes \cdots \otimes \sP_{N, \bx}$ over $\bx \sim \sX$).
    \end{enumerate}
\end{definition}

\begin{definition}[Good LVM]
    We call an LVM \emph{good} if, for all $\bx \in \Sigma$ and $i \in [N]$, $\sP_{i, \bx}$ is absolutely continuous with respect to $\sQ_i$ and $d\sP_{i, \bx} / d\sQ_i \in L^2(\sQ_i)$.
\end{definition}

\begin{remark}[Side information]
    We will not explore it further here, but one interesting setting that may be described if we let the channels vary over different coordinates is \emph{side information}, where we either reveal some coordinates of $\bx$ directly, or give in addition to $\by$ a vector correlated entrywise with $\bx$.
    See \cite{SN-2018-SBMSideInformation} for an example in the context of the stochastic block model, or \cite{DSMM-2018-ContextualSBM} for a ``contextual'' version of the model, which amounts to including a less direct form of side information.
\end{remark}

Over this broader class of models, we will prove the following result, which is a generalization of Theorem~\ref{thm:lvm}.

\begin{theorem}[Coordinate advantage for LVMs]
    \label{thm:lvm-general}
    Suppose $(\sX, \sQ, \sP)$ is a good LVM.
    Define the vector channel overlap
    \begin{equation}
        R_{\sQ, \sP, i}(\bx^{(1)}, \bx^{(2)}) \colonequals \Ex_{y \sim \sQ_i}\left[\left(\frac{d\sP_{i, \bx^{(1)}}}{d\sQ_i}(y) - 1\right)\left(\frac{d\sP_{i, \bx^{(2)}}}{d\sQ_i}(y) - 1\right)\right].
    \end{equation}
    Then,
    \begin{align}
      \CAdv(\sX, \sQ, \sP)^2
      &= \Ex_{\bx^{(1)}, \bx^{(2)} \sim \sX} \sum_{\substack{T \subseteq [N] \\ |T| \leq D}} \prod_{i \in T} R_{\sQ, \sP, i}(\bx^{(1)}, \bx^{(2)}) \\
      &\leq \Ex_{\bx^{(1)}, \bx^{(2)} \sim \sX} \exp^{\leq D}\left(\sum_{i = 1}^NR_{\sQ, \sP, i}(\bx^{(1)}, \bx^{(2)})\right).
    \end{align}
\end{theorem}

\subsection{Coordinate decomposition of likelihood ratio}

We will use the background results on coordinate or Efron-Stein decompositions developed in Appendix~\ref{sec:coord-general}.
Recall that the optimizer in the coordinate advantage is the LCDLR, the orthogonal projection of the likelihood ratio to $V_{\leq D}$, so it suffices to understand this object.
Lemma~\ref{lem:lr-decomp} describes this projection (in a setting generalizing LVMs) in terms of the \emph{marginal likelihood ratios}.

We first introduce notation for the coordinate-wise likelihood ratios,
\begin{equation}
    L_{i, \bx} \colonequals \frac{d\sP_{i, \bx}}{d\sQ}.
\end{equation}
Because $\by \sim \QQ$ has independent coordinates and, conditional on $\bx$, $\by \sim \PP$ also has independent coordinates, the likelihood ratio in an LVM may be computed as
\begin{equation}
    L(\by) \colonequals \frac{d\PP}{d\QQ}(\by) = \Ex_{\bx \sim \sX} \frac{d\PP_{\bx}}{d\QQ}(\by) = \Ex_{\bx \sim \sX} \prod_{i = 1}^N \frac{d\sP_{i, \bx}}{d\sQ_i}(y_i) = \Ex_{\bx \sim \sX} \prod_{i = 1}^N L_{i, \bx}(y_i).
\end{equation}
Following Appendix~\ref{sec:coord-general}, the marginal likelihood ratios are then
\begin{equation}
    L_T(\by) \colonequals \Ex_{y_i \sim \sQ_i : i \in [N] \setminus T} L(\by) = \Ex_{\bx \sim \sX} \prod_{i \in T} L_{i, \bx}(y_i) = \frac{d\PP_T}{d\QQ_T}(\by),
\end{equation}
where $\PP_T$ and $\QQ_T$ are the measures arising from an LVM formed by restricting the given LVM to the indices $T$.

Appendix~\ref{sec:coord-general} further describes a decomposition of $V_{\leq D}$ into mutually orthogonal subspaces $\what{V}_T$ over subsets $T \subseteq [N]$ with $|T| \leq D$, and by Lemma~\ref{lem:lr-decomp}, the projection of $L$ to $\what{V}_T$ is
\begin{equation}
    \what{L}_T(\by) \colonequals \sum_{S \subseteq T} (-1)^{|T| - |S|} L_S = \Ex_{\bx \sim \sX} \sum_{S \subseteq T} (-1)^{|T| - |S|} \prod_{i \in S} L_{i, \bx}(y_i) = \Ex_{\bx \sim \sX} \prod_{i \in T} (L_{i, \bx}(y_i) - 1).
\end{equation}
Indeed, by inclusion-exclusion it follows readily that $L = \sum_{T \subseteq [N]} \what{L}_T$, and we detail the orthogonality of this decomposition in Appendix~\ref{sec:coord-remarks}.
The LCDLR, the projection of $L$ to $V_{\leq D}$, which we will denote $L_{\leq D}$, is then
\begin{equation}
    L_{\leq D} = \sum_{\substack{T \subseteq [N] \\ |T| \leq D}} \what{L}_T = \sum_{\substack{T \subseteq [N] \\ |T| \leq D}}\Ex_{\bx \sim \sX} \prod_{i \in T} (L_{i, \bx}(y_i) - 1). \label{eq:lcdlr}
\end{equation}

\subsection{Coordinate advantage: Proof of Theorems \ref{thm:lvm} and \ref{thm:lvm-general}}

We will be ready to proceed to the main proof after one more preliminary step.
\begin{lemma}[Non-negativity]
    \label{lem:non-neg}
    For any $k_i \geq 0$ integers for $i \in [N]$, we have
    \begin{equation}
        \Ex_{\bx^{(1)}, \bx^{(2)} \sim \sX} \prod_{i = 1}^NR_{\sQ, \sP, i}(\bx^{(1)}, \bx^{(2)})^{k_i} \geq 0.
    \end{equation}
\end{lemma}
\begin{proof}
    We will just use that $R_{\sQ, \sP, i}$ is a product of inner products in $\sQ_i$, which may be viewed as a single inner product in $\QQ = \sQ_1 \otimes \cdots \otimes \sQ_N$, of the same function of $\bx^{(1)}$ and $\bx^{(2)}$.
    Explicitly, we may manipulate
    \begin{align*}
      &\Ex_{\bx^{(1)}, \bx^{(2)} \sim \sX} \prod_{i = 1}^NR_{\sQ, \sP, i}(\bx^{(1)}, \bx^{(2)})^{k_i} \\
      &= \Ex_{\bx^{(1)}, \bx^{(2)} \sim \sX} \prod_{i = 1}^N\left(\Ex_{y_i \sim \sQ_i}(L_{i, \bx^{(1)}}(y_i) - 1)(L_{i, \bx^{(2)}}(y_i) - 1)\right)^{k_i}
        \intertext{and writing a power of an expectation as an expectation of a product of independent copies,}
      &= \Ex_{\bx^{(1)}, \bx^{(2)} \sim \sX} \Ex_{\substack{y_i^{(1)}, \dots, y_i^{(k_i)} \sim \sQ_i \\ \text{for each } i \in [N]}}\prod_{i = 1}^N \prod_{j = 1}^{k_i}(L_{i, \bx^{(1)}}(y_i^{(j)}) - 1)(L_{i, \bx^{(2)}}(y_i^{(j)}) - 1) \\
      &= \Ex_{\substack{y_i^{(1)}, \dots, y_i^{(k_i)} \sim \sQ_i \\ \text{for each } i \in [N]}}\left(\Ex_{\bx \sim \sX}\prod_{i = 1}^N \prod_{j = 1}^{k_i}(L_{i, \bx^{(1)}}(y_i^{(j)}) - 1)\right)^2 \\
      &\geq 0, \numberthis
    \end{align*}
    completing the proof.
\end{proof}

\begin{proof}[Proof of Theorem~\ref{thm:lvm-general}]
    Recall that the coordinate advantage is the norm in $L^2(\QQ)$ of the LCDLR, which, by \eqref{eq:lcdlr} and the orthogonality of the $\what{L}_T$, is
    \begin{equation}
        \CAdv_{\leq D}(\sX, \sQ, \sP)^2 = \Ex_{\by \sim \QQ} L_{\leq D}(\by)^2 = \sum_{\substack{T \subseteq [N] \\ |T| \leq D}} \Ex_{\by \sim \QQ} \what{L}_T(\by)^2.
    \end{equation}
    Rewriting a squared expectation as an expectation of a product of two independent copies (as in the approach to the polynomial advantage of \cite{KWB-2022-LowDegreeNotes}),
    \begin{align*}
      \Ex_{\by \sim \QQ} \what{L}_T(\by)^2
      &= \Ex_{\by \sim \QQ} \what{L}_T(\by)^2 \\
      &= \Ex_{\by \sim \QQ} \left(\Ex_{\bx \sim \sX} \prod_{i \in T}\big( L_{i, \bx}(y_i) - 1 \big)\right)^2 \\
      &= \Ex_{\bx^{(1)}, \bx^{(2)} \sim \sX} \Ex_{\by \sim \QQ} \prod_{i \in T}\big( L_{i, \bx^{(1)}}(y_i) - 1 \big) \big( L_{i, \bx^{(2)}}(y_i) - 1 \big) \\
      &= \Ex_{\bx^{(1)}, \bx^{(2)} \sim \sX} \prod_{i \in T} R_{\sQ, \sP, i}(\bx^{(1)}, \bx^{(2)}), \numberthis
    \end{align*}
    where $\bx^{(1)}, \bx^{(2)} \sim \sX$ are independent draws.
    Summing over $T$ we find the first form of our result,
    \begin{align*}
      \CAdv_{\leq D}(\sX, \sQ, \sP)^2
      &= \Ex_{\bx^{(1)}, \bx^{(2)} \sim \sX} \sum_{\substack{T \subseteq [N] \\ |T| \leq D}} \prod_{i \in T} R_{\sQ, \sP, i}(\bx^{(1)}, \bx^{(2)})
  \intertext{and to obtain the bound we claim, note that the extra terms introduced when expanding the following expression by the binomial theorem are all non-negative upon taking the expectation over $\bx^{(1)}, \bx^{(2)}$ by Lemma~\ref{lem:non-neg}:}
  &\leq \Ex_{\bx^{(1)}, \bx^{(2)} \sim \sX} \sum_{d = 0}^D \frac{1}{d!} \left(\sum_{i = 1}^N R_{\sQ, \sP, i}(\bx^{(1)}, \bx^{(2)}) \right)^d \\
  &= \Ex_{\bx^{(1)}, \bx^{(2)} \sim \sX} \exp^{\leq D} \left(\sum_{i = 1}^N R_{\sQ, \sP, i}(\bx^{(1)}, \bx^{(2)})\right), \numberthis
\end{align*}
completing the proof.
\end{proof}

\section{Channel universality}

\subsection{Loose universality: Proof of Theorem~\ref{thm:channel-universality}}
\label{sec:pf:thm:channel-universality}

We now work under the more specific CLVM setting (Definition~\ref{def:clvm}).
We define notation for the likelihood ratios associated to the channel $\sP$ and recall the definition of the channel overlap, which may be rephrased in terms of these:
\begin{align}
  L_x &\colonequals \frac{d\sP_x}{d\sP_0}, \\
  R_{\sP}(x^{(1)}, x^{(2)}) &= \Ex_{y \sim \sP_0}\left[\left(\frac{d\sP_{x^{(1)}}}{d\sP_0}(y) - 1\right)\left(\frac{d\sP_{x^{(2)}}}{d\sP_0}(y) - 1\right)\right] \nonumber \\
      &= \Ex_{y \sim \sP_0}\left[\left(L_{x^{(1)}}(y) - 1\right)\left(L_{x^{(2)}}(y) - 1\right)\right] \nonumber \\
  &= \Ex_{y \sim \sP_0}\left[L_{x^{(1)}}(y)L_{x^{(2)}}(y)\right] - 1.
\end{align}
Let us first clarify the connection between our definition of the Fisher information and the more conventional one.
\begin{proposition}[Fisher information]
    \label{prop:fisher-info-equiv}
    Let $\sP$ be a channel so that, for each $y \in \Omega$, $L_x(y)$ is $\sC^1$ in $x$ in a neighborhood of $x = 0$.
    Then,
    \begin{equation}
        F_{\sP} = \frac{\partial^2R_{\sP}}{\partial x^{(1)} \partial x^{(2)}}(0, 0) = \Ex_{y \sim \sP_0} \left(\frac{\partial}{\partial x} L_x(y)\bigg|_{x = 0}\right)^2.
    \end{equation}
\end{proposition}
\noindent
This follows just by differentiating under the expectation twice.

We may also compute some other partial derivatives of $R_{\sP}$ directly.
\begin{proposition}
    \label{prop:channel-overlap-derivatives}
    Suppose that $R_{\sP}$ is $\sC^4$ in a neighborhood of the origin.
    Then,
    \begin{align}
      R_{\sP}(0, 0) &= 0, \\
      \frac{\partial^a R_{\sP}}{\partial x^{(1)^a}}(0, 0) = \frac{\partial^a R_{\sP}}{\partial x^{(2)^a}}(0, 0) &= 0 \text{ for all } a \in \{1, 2, 3, 4\}.
    \end{align}
\end{proposition}
\begin{proof}
    The first claim follows because $L_0 = d\sP_0 / d\sP_0 = 1$.
    The second claim follows because $R_{\sP}(x, 0) = \EE_{y \sim \sP_0} L_x(y) - 1 = \EE_{y \sim \sP_x} 1 - 1 = 0$ is a constant.
\end{proof}

The following bound on the truncated exponential polynomials, which is elementary but not trivial to show, will play an important role.
\begin{proposition}
    \label{prop:trunc-exp}
    Let $D \geq 2$ be even.
    Then, for all $x \in \RR$, we have
    \begin{equation}
        0 < \exp^{\leq D}(x) \leq \exp(|x|),
    \end{equation}
    and, for all $x, y \in \RR$, we have
    \begin{equation}
        \frac{\exp^{\leq D}(x)}{\exp(100|y|)} \leq \exp^{\leq D}(x + y) \leq \exp^{\leq D}(x) \exp(100|y|).
    \end{equation}
\end{proposition}
\noindent
We have not attempted to optimize the constant 100.
We defer the proof to Appendix~\ref{app:prop:trunc-exp}.
Intuitively, the latter inequality is a variation on the exponential identity $\exp(x + y) = \exp(x)\exp(y)$; one may check that the direct analog $\exp^{\leq D}(x + y) \leq \exp^{\leq D}(x)\exp^{\leq D}(y)$ unfortunately does not hold in general.

\begin{proof}[Proof of Theorem~\ref{thm:channel-universality}]
    The main idea is to approximate $R_{\sP}(x^{(1)}, x^{(2)}) \approx F_{\sP} x^{(1)}x^{(2)}$.
    Thus, let us define the remainder in this approximation:
    \begin{equation}
      \Delta(x^{(1)}, x^{(2)}) \colonequals R_{\sP}(x^{(1)}, x^{(2)}) - F_{\sP} x^{(1)}x^{(2)}.
  \end{equation}
  We first bound this remainder.
  By Assumption C1, $R_{\sP}$ is $\sC^4$ in an open set $U \supset [-A, A]^2$.
  Thus, consider the partial derivatives of $R_{\sP}$ at the origin up to order 3.
  By Assumption C2 and Proposition~\ref{prop:channel-overlap-derivatives}, the only non-zero one of these derivatives is the mixed second partial derivative, whose value is $F_{\sP}$ by definition.
  So, by the multivariate Taylor theorem with the Lagrange form of the remainder, for each $(x^{(1)}, x^{(2)}) \in [-A, A]^2$, there exist some $z^{(1)}, z^{(2)} \in [-A, A]$ such that
  \begin{equation}
      \Delta(x^{(1)}, x^{(2)}) = \sum_{\substack{k, \ell \geq 1 \\ k + \ell = 4}} \frac{1}{k! \, \ell!}\frac{\partial^4 R_{\sP}}{\partial x^{(1)^k} \partial x^{(2)^{\ell}}}(z^{(1)}, z^{(2)}) x^{(1)^k} x^{(2)^{\ell}}.
  \end{equation}
  Define the following, which is finite because $R_{\sP}$ is $\sC^4$ on $U$:
  \begin{equation}
      K = K(\sP) \colonequals \max_{\substack{k, \ell \geq 1 \\ k + \ell = 4}} \max_{z^{(1)}, z^{(2)} \in [-A, A]^2} \left|\frac{\partial^4 R_{\sP}}{\partial x^{(1)^k} \partial x^{(2)^{\ell}}}(z^{(1)}, z^{(2)})\right|.
  \end{equation}
  We may bound
  \begin{equation}
      |\Delta(x^{(1)}, x^{(2)})| \leq K \sum_{\substack{k, \ell \geq 1 \\ k + \ell = 4}} |x^{(1)}|^k |x^{(2)}|^{\ell}
  \end{equation}

  Now, we consider the sum of overlaps appearing in the upper bound of Theorem~\ref{thm:lvm}:
  \begin{equation}
      \sum_{i = 1}^N R_{\sP}(x_i^{(1)}, x_i^{(2)}) = F_{\sP} \langle \bx^{(1)}, \bx^{(2)} \rangle + \underbrace{\sum_{i = 1}^N \Delta(x^{(1)}, x^{(2)})}_{\colonequals \Delta(\bx^{(1)}, \bx^{(2)})}.
  \end{equation}
  We may bound this summed error term by
  \begin{align*}
    |\Delta(\bx^{(1)}, \bx^{(2)})|
    &\leq \sum_{i = 1}^N |\Delta(x^{(1)}_i, x^{(2)}_i)| \\
    &\leq K\sum_{\substack{k, \ell \geq 1 \\ k + \ell = 4}} \sum_{i = 1}^N |x^{(1)}_i|^k |x^{(2)}_i|^{\ell}
      \intertext{and by the Cauchy-Schwarz inequality,}
    &\leq K\sum_{\substack{k, \ell \geq 1 \\ k + \ell = 4}} \|\bx^{(1)}\|_{2k}^k \|\bx^{(2)}\|_{2\ell}^{\ell}
    \intertext{and by Assumption P2, when $\bx^{(1)}, \bx^{(2)} \sim \sX$, almost surely we have}
    &\leq K\sum_{\substack{k, \ell \geq 1 \\ k + \ell = 4}} B_{2k}^k B_{2\ell}^{\ell} N^{1 - \frac{k + \ell}{4}} \\
    &= K\sum_{\substack{k, \ell \geq 1 \\ k + \ell = 4}} B_{2k}^k B_{2\ell}^{\ell} \\
    &\colonequals C \numberthis
  \end{align*}
  which is a constant $C = C(\sP)$ depending only on $\sP, B_2, B_4, B_6$.
  In summary, we have found that, almost surely when $\bx^{(1)}, \bx^{(2)} \sim \sX$, we have
  \begin{equation}
      \left|\sum_{i = 1}^N R_{\sP}(x_i^{(1)}, x_i^{(2)}) - F_{\sP} \langle \bx^{(1)}, \bx^{(2)} \rangle\right| \leq C.
      \label{eq:univ-overlap-bound}
  \end{equation}

  Recall that we must prove both upper and lower bounds on the coordinate advantage.
  The upper bound follows by plugging \eqref{eq:univ-overlap-bound} into the upper bound of Theorem~\ref{thm:lvm} and then applying Proposition~\ref{prop:trunc-exp}.

  For the lower bound, we first make a general calculation giving a lower bound counterpart to Theorem~\ref{thm:lvm}.
  For the sake of conciseness, in an expectation over $\bx^{(1)}, \bx^{(2)}$, let us write $R_i \colonequals R_{\sP}(x_i^{(1)}, x_i^{(2)})$.
  We have:
  \begin{align*}
    &\Ex_{\bx^{(1)}, \bx^{(2)} \sim \sX}\left[\exp^{\leq D}\left(\sum_{i = 1}^N R_i\right) - \left(\sum_{i = 1}^N R_i^2\right) \exp^{\leq D - 2}\left(\sum_{i = 1}^N R_i\right)\right] \\
    &= 1 + \Ex_{\bx^{(1)}, \bx^{(2)} \sim \sX} \sum_{i = 1}^N R_i + \sum_{d = 2}^D\Ex_{\bx^{(1)}, \bx^{(2)} \sim \sX}\left[\frac{1}{d!}\left(\sum_{i = 1}^N R_i\right)^d - \frac{1}{(d - 2)!} \left(\sum_{i = 1}^N R_i^2\right)\left(\sum_{i = 1}^N R_i\right)^{d - 2}\right]
      \intertext{and here, expanding the powers inside, the coefficient of a given $\prod_{i = 1}^N R_i^{k_i}$ with some $k_i \geq 2$ will be at most $\frac{1}{k_1! \cdots k_N!} - \frac{k_i!}{k_1! \cdots k_N!} \leq 0$. Thus, together with Lemma~\ref{lem:non-neg}, we find}
    &\leq 1 + \Ex_{\bx^{(1)}, \bx^{(2)} \sim \sX} \sum_{i = 1}^N R_i + \sum_{d = 2}^D \Ex_{\bx^{(1)}, \bx^{(2)} \sim \sX} \sum_{\substack{T \subseteq [N] \\ |T| = d}} \prod_{i \in T} R_i \\
    &= \CAdv_{\leq D}(\sX, \sP). \numberthis
  \end{align*}
  by the formula of Theorem~\ref{thm:lvm}.
  By \eqref{eq:univ-overlap-bound}, we have almost surely when $\bx^{(1)}, \bx^{(2)} \sim \sX$ that
  \begin{align*}
    \sum_{i = 1}^N R_i^2
    &= \sum_{i = 1}^N R_{\sP}(x^{(1)}_i, x^{(2)}_i)^2 \\
    &\leq 2F_{\sP}^2 \sum_{i = 1}^N x_i^{(1)^2} x_i^{(2)^2} + \sum_{i = 1}^N \Delta(x^{(1)}_i, x^{(2)}_i)^2 \\
    &\leq 2F_{\sP}^2 \|\bx^{(1)}\|_4^2\|\bx^{(2)}\|_4^2 + 3K^2 \sum_{\substack{k, \ell \geq 1 \\ k + \ell = 4}} \sum_{i = 1}^N |x^{(1)}_i|^{2k}|x^{(2)}_i|^{2\ell} \\
    &\leq 2F_{\sP}^2 \|\bx^{(1)}\|_4^2\|\bx^{(2)}\|_4^2 + 3K^2 \sum_{\substack{k, \ell \geq 1 \\ k + \ell = 4}} \|\bx^{(1)}\|_{4k}^{2k} \|\bx^{(2)}\|_{4\ell}^{2\ell}, \numberthis
  \end{align*}
  and bounding these norms with the assumptions on $\sX$ as above shows that this bounded by a constant depending only on $\sP$.
  The proof of the lower bound is then completed by using Proposition~\ref{prop:trunc-exp} twice more.
\end{proof}

\begin{remark}[Variable channels]
    \label{rem:different-channels}
    It is straightforward to extend the result to allow for different channels $\sP_i$ per coordinate, with different Fisher informations $F_{\sP_i}$, so long as the fourth mixed partial derivatives of the associated channel overlaps $R_{\sP_i}$ are uniformly bounded on an open set containing $[-A, A]$, where $A$ is as in Assumption P1 on $\sX$.
    In this case, $F_{\sP} \langle \bx^{(1)}, \bx^{(2)} \rangle$ would be replaced by $\bx^{(1)^{\top}} \bF \bx^{(2)}$ for $\bF$ the diagonal matrix of the $F_{\sP_i}$.
    More generally, it seems plausible that similar results should hold even if $y_i$ can depend on all of $\bx$ rather than just $x_i$, so that the channels $\sP_{i, \bx}$ depend on the entire vector $\bx$ (as in a general LVM in the setting of Theorem~\ref{thm:lvm-general}).
    The overlap expression would then involve the \emph{Fisher information matrix} of the channel, part of the Hessian matrix of the now multivariate analog of $R_{\sP}$.
    We do not pursue this here as it would require more complicated conditions on the prior $\sX$ and would make our calculations more elaborate.
\end{remark}

\subsection{Tight universality: Proof of Theorem~\ref{thm:dilution}}
\label{sec:pf:thm:dilution}

The following is immediate from the definition of dilution.
\begin{proposition}[Overlap invariance]
    For any $\sX$ and any $k \geq 1$, the law of $\langle \bx^{(1)}, \bx^{(2)} \rangle$ is the same for $\bx^{(1)}, \bx^{(2)} \sim \sX$ independently as for $\bx^{(1)}, \bx^{(2)} \sim \sD_k \sX$ independently.
\end{proposition}
\noindent
We will also use the following simple combinatorial fact, proved in Appendix~\ref{app:pf:prob:binom-lb}.

\begin{proposition}
    \label{prop:binom-lb}
    Suppose $t \leq k / 2$. Then, $\binom{k}{t} \geq \frac{k^t}{t!} \exp(-\frac{t^2}{k})$.
\end{proposition}

\begin{proof}[Proof of Theorem~\ref{thm:dilution}]
    Let us first sketch the proof ideas.
    The simple main idea is that dilution makes the original upper bound on the coordinate advantage from Theorem~\ref{thm:lvm} (not involving universality) close to tight.
    This is a general property of the truncated exponential polynomials $\exp^{\leq D}$ and does not rely on any special properties of the channel or prior.
    We then note that the proof of Theorem~\ref{thm:channel-universality} (showing loose channel universality) really proceeded by showing that \emph{this bound} behaves universally.
    Thus, if the bound is close to tight, then the coordinate advantage itself also behaves universally.

    At a technical level, Theorem~\ref{thm:channel-universality} already proves the upper bound on the coordinate advantage that we need, so it suffices to show a matching lower bound.
    Note that, under our assumption, $k \geq D^2 \geq D$.
    We have:
\begin{align*}
  \CAdv_{\leq D}(\sD_k \sX, \sP_x)
  &= \Ex_{\widetilde{\bx}^1, \widetilde{\bx}^2 \sim \sD_k \sX} \sum_{\substack{T \subseteq [kN] \\ |T| \leq D}} \prod_{i \in T} R_{\sP}(\widetilde{x}_i^1, \widetilde{x}_i^2) \\
  &= \Ex_{\bx^1, \bx^2 \sim \sX} \sum_{\substack{\bt \in \NN^N \\ |\bt| \leq D}} \prod_{i = 1}^N \binom{k}{t_i} R_{\sP}\left(\frac{x_i^1}{\sqrt{k}}, \frac{x_i^2}{\sqrt{k}}\right)^{t_i}
  \intertext{where $|t| \colonequals \sum_{i = 1}^N t_i$ and we should also constrain $|t_i| \leq k$, but since $k \geq D$ this is vacuous. Using Proposition~\ref{prop:binom-lb} and that $t_i \leq D$ while $k \geq D^2$,}
  &\geq \frac{1}{e}\Ex_{\bx^1, \bx^2 \sim \sX} \sum_{\substack{\bt \in \NN^N \\ |\bt| \leq D}} \prod_{i = 1}^N \frac{1}{t_i!} \left(k R_{\sP}\left(\frac{x_i^1}{\sqrt{k}}, \frac{x_i^2}{\sqrt{k}}\right)\right)^{t_i} \\
  &= \frac{1}{e}\Ex_{\bx^1, \bx^2 \sim \sX} \sum_{d = 0}^D \frac{1}{d!}\sum_{\substack{\bt \in \NN^N \\ |\bt| = d}} \binom{d}{t_1 \cdots t_N} \left(k R_{\sP}\left(\frac{x_i^1}{\sqrt{k}}, \frac{x_i^2}{\sqrt{k}}\right)\right)^{t_i}
  \intertext{and by the multinomial theorem,}
  &= \frac{1}{e}\Ex_{\bx^1, \bx^2 \sim \sX} \sum_{d = 0}^D \frac{1}{d!}\left(k\sum_{i = 1}^N R_{\sP}\left(\frac{x_i^1}{\sqrt{k}}, \frac{x_i^2}{\sqrt{k}}\right)\right)^d \\
  &= \frac{1}{e}\Ex_{\bx^1, \bx^2 \sim \sX} \exp^{\leq D}\left(k \sum_{i = 1}^N R_{\sP}\left(\frac{x_i^1}{\sqrt{k}}, \frac{x_i^2}{\sqrt{k}}\right)\right)
    \intertext{and the argument of Theorem~\ref{thm:channel-universality}, repeated \emph{mutatis mutandis}, noting that powers of $k$ cancel in the Taylor expansion of $R_{\sP}$, gives}
  &\geq C\Ex_{\bx^1, \bx^2 \sim \sX} \exp^{\leq D}\left(F_{\sP} \langle \bx_i^1, \bx_i^2 \rangle \right) \\
  &= C\Ex_{\widetilde{\bx}^1, \widetilde{\bx}^2 \sim \sD_k\sX} \exp^{\leq D}\left(F_{\sP} \langle \widetilde{\bx}_i^1, \widetilde{\bx}_i^2 \rangle \right) \numberthis
\end{align*}
for $C$ a constant depending only on the channel and the bounds on the prior, completing the proof.
\end{proof}

Finally, let us elaborate on the Introduction and give some heuristic discussion of why dilute models might behave like Gaussian ones.
Recall that the \Cramer-Rao inequality gives a limitation on the variance of estimators of $x$ from $y$ in CLVMs (see, e.g., Chapter 5 of \cite{Pitman-1979-BasicTheoryStatisticalInference}).

\begin{proposition}[Cram\'{e}r-Rao inequality]
    \label{prop:cramer-rao}
    In a CLVM, let $f: \Omega \to \Sigma$ be a measurable function such that $\mu_f(x) \colonequals \EE_{y \sim \sP_x} f(y)$ is $\sC^1$.
    Then, $\Var_{y \sim \sP_0} f(y) \geq \mu_f^{\prime}(0)^2 / F_{\sP}$.
\end{proposition}
\noindent
Usually one defines an entire Fisher information function $F_{\sP}(x)$ analogously to our definition, from which one obtains similar lower bounds on the variance of any estimator under any $\sP_x$.
We also mention the following ancillary result, to clarify the discussion in Remark~\ref{rem:exp-hardest}.
\begin{corollary}
    \label{cor:fisher-info-lb}
    Suppose that the variance of $\sP_0$ is $\sigma^2$, and the $\sP_x$ are parametrized such that the mean of $\sP_x$ is $x$.
    Then, $F_{\sP} \geq 1 / \sigma^2$.
\end{corollary}
\begin{proof}
    In Proposition~\ref{prop:cramer-rao}, take $f(y) = y$, which has $\mu_f(x) = x$ by the assumption.
\end{proof}

We may produce a function that, at least locally near $x = 0$, is an unbiased estimator and saturates the inequality.
In a CLVM, consider the function
\begin{equation}
    f(y) = \frac{1}{F_{\sP}} \frac{\partial}{\partial x}L_x(y) \bigg|_{x = 0},
\end{equation}
sometimes called the \emph{Fisher score}.
Like with the Fisher information, though, the Fisher score can be evaluated at arbitrary $x$ while we are interested only in evaluating it at $x = 0$, so we call $f(y)$ the \emph{local Fisher score} instead.
For example, the nonlinearity used in the pretransformed eigenvalue test of Proposition~\ref{prop:pwbm} is the special case of the local Fisher score for additive noise models.
The following is one conceptual justification for using the local Fisher score.

\begin{remark}[Approximate maximum likelihood]
    \label{rem:fisher-score-mle}
    One interpretation of the local Fisher score is as an approximation of the maximum likelihood estimator of $x$.
    Indeed, consider the log-likelihood ratio $\ell_x(y) \colonequals \log L_x(y)$.
    One may check, under mild regularity conditions, that an alternative formula for the Fisher information is $F_{\sP} = -\frac{\partial^2}{\partial x^2} \ell_x(y) |_{x = 0}$, and so $f(y) = -\frac{\partial}{\partial x} \ell_x(y)|_{x = 0} \, / \frac{\partial^2}{\partial x^2} \ell_x(y) |_{x = 0}$.
    This means that $f(y)$ gives the result of taking one step of Newton's method for maximizing the log-likelihood in $x$, starting from $x = 0$.
\end{remark}

The local Fisher score is ``locally unbiased'' near $x = 0$, in the sense that
\begin{align}
  \Ex_{y \sim \sP_0} f(y) &= \frac{1}{F_{\sP}} \frac{\partial R_{\sP}}{\partial x^{(1)}}(0, 0) = 0, \\
  \frac{\partial}{\partial x} \Ex_{y \sim \sP_x} f(y) \bigg|_{x = 0} &= \frac{1}{F_{\sP}} \frac{\partial^2 R_{\sP}}{\partial x^{(1)}\partial x^{(2)}}(0, 0) = \frac{1}{F_{\sP}} \cdot F_{\sP} = 1,
\end{align}
so $\Ex_{y \sim \sP_x} f(y) \approx x$ to leading order for small $x$.
And, we have
\begin{equation}
    \Var_{y \sim \sP_0} f(y)^2 = \Ex_{y \sim \sP_0} f(y)^2 = \frac{1}{F_{\sP}^2} \Ex_{y \sim \sP_0}\left(\frac{\partial}{\partial x}L_x(y) \bigg|_{x = 0}\right)^2 = \frac{1}{F_{\sP}^2} \cdot F_{\sP} = \frac{1}{F_{\sP}},
\end{equation}
so $f(y)$ saturates the \Cramer-Rao inequality and may be viewed as a ``minimum variance locally unbiased estimator'' (paralleling the more common notion of minimum variance unbiased estimator, which asks for exact unbiasedness over all $x$).

For small $x$, by continuity, we expect the variance of the local Fisher score over $\sP_x$ to still be $\Var_{y \sim \sP_x} f(y)^2 = 1 / F_{\sP} + O(x)$.
Turning to the application to diluted models, suppose we have access to $y_1, \dots, y_k \sim \sP_{x / \sqrt{k}}$ that are i.i.d.
By the above observations, the random variable
\begin{equation}
    \frac{1}{\sqrt{k}} \sum_{i = 1}^k f(y_i) = x + \frac{1}{\sqrt{k}} \sum_{i = 1}^k \left(f(y_i) - \frac{x}{\sqrt{k}}\right)
\end{equation}
will, for large $k$, have law close to $\sN(x, 1 / F_{\sP})$ by the central limit theorem.
(As in Remark~\ref{rem:fisher-score-mle}, this estimator is approximating the maximum likelihood estimate of $x$ from the samples $y_i$, and this limit is comparable to the asymptotic normality of the maximum likelihood estimator.)
In this way, it is always possible to reduce dilute observations through an arbitrary channel to less dilute observations through an (approximately) additive Gaussian channel.
By the \Cramer-Rao inequality, the variance $1 / F_{\sP}$ is optimal, in the sense that there is no transformation $f$ that through the same operations yields approximately Gaussian observations with mean $x$ and with smaller variance (as also verified by hand in a special case using the calculus of variations by \cite{PWBM-2018-PCAI}).

\section{Applications}

\subsection{Spiked matrix models: Proof of Corollary~\ref{cor:non-gaussian-smm}}
\label{sec:spiked-mx}

We will use the following result coming from the analysis of a Gaussian spiked matrix model.
\begin{proposition}[Theorem 3.9 of \cite{KWB-2022-LowDegreeNotes}]
    \label{prop:kwb-spiked-wigner}
    Under the assumptions of Corollary~\ref{cor:non-gaussian-smm}, for any $D = D(n) = o(n / \log n)$,
    \begin{equation}
        \Ex_{\bx^{(1)}, \bx^{(2)} \sim \pi^{\otimes n}} \exp^{\leq D}\left(\frac{\lambda^2}{2n}\langle \bx^{(1)}\bx^{(1)^{\top}}, \bx^{(2)}\bx^{(2)^{\top}}\rangle\right) = O(1).
    \end{equation}
\end{proposition}
\noindent
This appears as a bound on the polynomial advantage through an additive Gaussian channel where the variance of the diagonal noise is exactly twice the variance of the off-diagonal noise, so that the symmetric noise matrix is drawn from the Gaussian orthogonal ensemble.

\begin{proof}[Proof of Corollary~\ref{cor:non-gaussian-smm}]
    The upper bound on the coordinate advantage follows from our Theorem~\ref{thm:channel-universality} and Proposition~\ref{prop:kwb-spiked-wigner}, as follows.
    Note that under our assumptions, $\sX_n$ satisfies Assumptions P1 and P2 of Theorem~\ref{thm:channel-universality} and $\sP$ (the additive noise channel for a noise density $p$) satisfies Assumptions C1 and C2 by Proposition~\ref{prop:channel-additive}.
    Let $K$ be a number such that $x \sim \pi$ has $|x| \leq K$ almost surely.
    Let us write $\bX^{(1)}, \bX^{(2)} \sim \sX_n$ for two independent draws from the prior described in the statement, so that $\bX^{(i)}$ contains the upper triangle of $\frac{\lambda}{\sqrt{n}}\bx^{(i)}\bx^{(i)^{\top}}$ for $\bx^{(i)}$ having i.i.d.\ entries drawn from $\pi$.
    We assume without loss of generality that $D$ is even.
    Then, Theorem~\ref{thm:channel-universality} gives, for a constant $C$ depending only on the channel and the bounds on the prior,
    \begin{align*}
      \CAdv_{\leq D}(\sX_n, \sP)
      &\leq C \Ex_{\bX^{(1)}, \bX^{(2)} \sim \sX_n} \exp^{\leq D}\left(\frac{1}{F_{\sP}} \langle \bX^{(1)}, \bX^{(2)}\rangle\right) \\
      &= C \Ex_{\bx^{(1)}, \bx^{(2)} \sim \pi^{\otimes n}} \exp^{\leq D}\left(\frac{\lambda^2}{2nF_{\sP}} \left(\langle \bx^{(1)} \bx^{(1)^{\top}}, \bx^{(2)} \bx^{(2)^{\top}}\rangle - \sum_{i = 1}^n|x_i^{(1)}|^2|x_i^{(2)}|^2\right) \right) \\
      &\leq C \Ex_{\bx^{(1)}, \bx^{(2)} \sim \pi^{\otimes n}} \exp^{\leq D}\left(\frac{\lambda^2}{2nF_{\sP}} \langle \bx^{(1)} \bx^{(1)^{\top}}, \bx^{(2)} \bx^{(2)^{\top}}\rangle + \frac{\lambda^2}{2F_{\sP}}K^4 \right) \numberthis
    \end{align*}
    and using Proposition~\ref{prop:trunc-exp} together with Proposition~\ref{prop:kwb-spiked-wigner} shows that this is $O(1)$ so long as $\lambda < 1 / \sqrt{F_{\sP}}$, as claimed.

    For the lower bound on the coordinate advantage when $\lambda > 1 / \sqrt{F_{\sP}}$, we note that we cannot use the lower bound from Theorem~\ref{thm:channel-universality}, which would give a threshold for $\lambda$ off by a constant factor.
    We could appeal to Theorem~\ref{thm:dilution}, but this would require dilution of the prior.
    Instead, we make a more hands-on argument, producing an explicit function of low coordinate degree that witnesses that the coordinate advantage is large.
    Namely, let $f(y) \colonequals -p^{\prime}(y) / p(y)$ be the local Fisher score as in Proposition~\ref{prop:pwbm}, and define for a symmetric matrix $\bY$
    \begin{equation}
        g(\bY) \colonequals \Tr\left(\left(\frac{1}{\sqrt{n}}f(\bY)\right)^D\right),
    \end{equation}
    where $f(\bY)$ is applied entrywise.
    By expanding the trace, we see that this has $\cdeg(g) \leq D$.
    We have assumed $D = \omega(\log n)$, but let us also assume that $D < n$; this is without loss of generality, since the coordinate advantage is increasing in $D$.

    We now control the first moment of this under $\PP$ and the second moment under $\QQ$.
    We have, by Theorem 4.8 of \cite{PWBM-2018-PCAI},
    \begin{align*}
      \Ex_{\bY \sim \PP} g(\bY)
      &\geq \Ex_{\bY \sim \PP} \lambda_{\max}\left(\frac{1}{\sqrt{n}}f(\bY)\right)^D \\
      &\geq (1 - o(1)) \left(\lambda F_{\sP} + \frac{1}{\lambda}\right)^D \\
      &\geq (1 - o(1)) \left(2\sqrt{F_{\sP}} + \epsilon\right)^D, \numberthis
    \end{align*}
    for some $\epsilon > 0$ depending on $\lambda$.
    And, we have
    \begin{align*}
      \Ex_{\bY \sim \QQ} g(\bY)^2
      &= \Ex_{\bY \sim \QQ} \left(\Tr\left(\left(\frac{1}{\sqrt{n}}f(\bY)\right)^D\right)\right)^2 \\
      &\leq n \Ex_{\bY \sim \QQ} \Tr\left(\left(\frac{1}{\sqrt{n}}f(\bY)\right)^{2D}\right)
        \intertext{and here, noting that $f(\bY)$ is a Wigner matrix with bounded i.i.d.\ entries that are centered and have variance $F_{\sP}$, standard combinatorial analysis of Wigner matrices following \cite{FK-1981-EigenvaluesRandomMatrices} (see, e.g., Section~2.1.6 of \cite{AGZ-2010-RandomMatrices}) gives, for $D < n$,}
      &\leq 2n^2 (2\sqrt{F_{\sP}})^{2D}. \numberthis
    \end{align*}

    Thus, the coordinate advantage is by definition bounded below as
    \begin{equation}
        \CAdv_{\leq D}(\sX_n, \sP) \geq \frac{\EE_{\bY \sim \PP} \, g(\bY)}{\sqrt{\EE_{\bY \sim \QQ} \, g(\bY)^2}} \geq \frac{(1 - o(1))}{2n^2}\left(1 + \frac{\epsilon}{2\sqrt{F_{\sP}}}\right)^{D},
    \end{equation}
    and the result follows since $D = \omega(\log n)$, so this diverges as $D \to \infty$.
\end{proof}

\begin{remark}
    By similar arguments one may allow for the diagonal of $\bx^{(i)}\bx^{(i)^{\top}}$ to be included in $\bX^{(i)}$, or, using the more general bound outlined in Remark~\ref{rem:different-channels}, for the diagonal to be included with additive noise having twice the variance, so that the noise matrix is a Wigner matrix with the usual scaling, or for that matter for the diagonal to be included with any other bounded noise distribution.
\end{remark}

\subsection{Spiked tensor models: Proof of Corollary~\ref{cor:non-gaussian-stm}}

We will need the following slight elaboration on the second claim of Proposition~\ref{prop:kwb-stm}, which is obtained from examining the proofs of \cite{KWB-2022-LowDegreeNotes,Kunisky-2021-SpectralBarriersCertification}.
We give more details in Appendix~\ref{app:pf:prop:kwb-stm-2}.
\begin{proposition}
    \label{prop:kwb-stm-2}
    Let $\sP_x = \sN(x, 1)$, let $\sX$ be as in Proposition~\ref{prop:kwb-stm}, and let $K > 0$.
    Then, there exists $b_q > 0$ depending on $K$ such that, if $\lambda \geq b_q D^{-(q - 2)/4}$ and $D$ is even with $2 \leq D \leq \frac{2}{q}n$, then
    \begin{equation}
        \frac{\Adv_{\leq D}(\sX_n, \sP)}{\Adv_{\leq D - 2}(\sX_n, \sP)} \geq K.
    \end{equation}
\end{proposition}

\begin{proof}[Proof of Corollary~\ref{cor:non-gaussian-stm}]
    The first claim follows immediately by using the upper bound of Theorem~\ref{thm:channel-universality} on the coordinate advantage and applying the first claim of Proposition~\ref{prop:kwb-stm}:
    \begin{equation}
        \CAdv_{\leq D}(\sX, \sP)^2 \leq C_1\,\Adv_{\leq D}(\sX, \sN(0, 1 / F_{\sP}))^2,
    \end{equation}
    where if $a_q$ is the constant from the Proposition, then taking $a_{q, \sP} \colonequals a_q / \sqrt{F_{\sP}}$ absorbs the extra factor of $F_{\sP}$ from the right-hand side.

    For the second claim, we use the lower bound of Theorem~\ref{thm:channel-universality} and Proposition~\ref{prop:kwb-stm-2}: taking $b_{q, \sP}$ sufficiently large, by the Proposition we will have for all even $D \leq \frac{2}{q}n$ that
    \begin{align*}
      \CAdv_{\leq D}(\sX, \sP)^2
      &\geq C_2\,\Adv_{\leq D}(\sX, \sN(0, 1 / F_{\sP}))^2 - C_3\,\Adv_{\leq D - 2}(\sX, \sN(0, 1 / F_{\sP}))^2 \\
      &\geq \frac{C_2}{2} \Adv_{\leq D}(\sX, \sN(0, 1 / F_{\sP}))^2 \\
      &\geq \frac{C_2}{2}K^D. \numberthis
    \end{align*}
    We may then remove the constraint that $D$ is even by observing that the coordinate advantage is monotonically increasing in $D$.
\end{proof}

\subsection{Censorship: Proof of Theorem~\ref{thm:censorship}}
\label{sec:censorship}

The proof will follow by a simple calculation.

\begin{proof}[Proof of Theorem~\ref{thm:censorship}]
First, the channel likelihood ratio after censorship is given by:
\begin{equation}
    L_x^{\sC_{\eta}}(y) \colonequals \frac{d\sC_{\eta} \sP_x}{d\sC_{\eta} \sP_0}(y) = \left\{\begin{array}{ll} 1 & \text{if } y = \bullet, \\ L_x(y) & \text{otherwise.} \end{array}\right.
\end{equation}
We note that this does not depend on $\eta$.
And, the channel overlap is:
\begin{equation}
    R_{\sC_{\eta}\sP}(x^{(1)}, x^{(2)}) = \Ex_{y \sim \sC_{\eta}\sP_0} (L_{x^{(1)}}^{\sC_{\eta}}(y) - 1)(L_{x^{(2)}}^{\sC_{\eta}}(y) - 1) = (1 - \eta)R_{\sP}(x^{(1)}, x^{(2)}).
\end{equation}
Clearly this scaling does not affect Assumptions C1 and C2, and the Theorem is proved.
\end{proof}

We described in the Introduction the application of this result to spiked matrix models.
This gives a computational lower bound, but it is not obvious what a concrete matching polynomial time algorithm would be.

There is, however, a natural guess, extending the strategy of \cite{PWBM-2018-PCAI} to non-additive noise and as predicted by \cite{LKZ-2015-LowRankChannelUniversality}.
Namely, per our discussion in Section~\ref{sec:pf:thm:dilution}, intuitively we should expect the entrywise local Fisher score function $f(y) = \frac{\partial}{\partial x} L_x(y) |_{x = 0}$ to roughly map general noise models to additive Gaussian models.
Computing this for the censored likelihood ratio appearing above, we find that $f(\bullet) = 0$ while $f(y)$ for $y \neq \bullet$ is the same as the local Fisher score in the uncensored model.
We thus arrive at a natural strategy for working with censored data in CLVMs: set the censored entries to zero and proceed as though given an observation from an uncensored model.
For spiked matrix models, we reach the following conjecture in random matrix theory, which is simple to verify numerically but seems non-trivial to attack using standard techniques as in, e.g., \cite{FP-2007-LargestEigenvalueWigner,CDMF-2009-DeformedWigner}, since the additive structure of the model has been corrupted.

\begin{conjecture}
    \label{conj:censored-rmt}
    Let $\bx$ be random as in Assumption~\ref{ass:spike-prior}, $p(y)$ a density as in Proposition~\ref{prop:channel-additive}, $\sP$ the corresponding additive noise channel, $\eta \in (0, 1)$, and define $f(y) \colonequals -p^{\prime}(y) / p(y)$.
    Write $f(\bY)$ for the entrywise application of $f$ to a matrix $\bY$.
    Let $\bY^{(0)} \colonequals \frac{\lambda}{\sqrt{n}}\bx\bx^{\top} + \bW$ for $W_{ij} = W_{ji} \sim \sP_0$ i.i.d.\ (i.e., drawn with density $p$), and $W_{ii} = 0$.
    Let $\bY$ be formed by replacing every entry in the upper triangle of $\bY^{(0)}$ with zero independently with probability $\eta$, and repeating the same replacements symmetrically in the lower triangle.
    We conjecture that there exists some $\gamma \in \RR$ depending only on $p$ and $\eta$ such that:
    \begin{enumerate}
    \item If $\lambda < 1 / \sqrt{(1 - \eta)F_{\sP}}$, then $\frac{1}{\sqrt{n}} \lambda_{\max}(f(\bY)) \to \gamma$ in probability.
    \item If $\lambda > 1 / \sqrt{(1 - \eta)F_{\sP}}$, then there is $\epsilon = \epsilon(\lambda) > 0$ such that $\frac{1}{\sqrt{n}} \lambda_{\max}(f(\bY)) \to \gamma + \epsilon$ in probability.
    \end{enumerate}
\end{conjecture}

\subsection{Quantization: Proof of Theorem~\ref{thm:quantization}}
\label{sec:quantization}

We may prove our main result on quantization immediately.
\begin{proof}[Proof of Theorem~\ref{thm:quantization}]
    The result again follows by straightforward calculations.
    The channel likelihood ratio $L_x^{\sgn} \colonequals d\sgn(\sP)_x / d\sgn(\sP)_0$ takes values
    \begin{align}
      L_x^{\sgn}(1) &= \frac{\sP_x[\sgn(y) = 1]}{\sP_0[\sgn(y) = 1]} = 2\int_{-x}^{\infty} p(y)\,dy, \\
      L_x^{\sgn}(-1) &= \frac{\sP_x[\sgn(y) = -1]}{\sP_0[\sgn(y) = -1]} = 2\int_{-\infty}^{-x} p(y)\,dy,
    \end{align}
    where the latter formulas follow since we assume that $p$ is symmetric, so that $\sgn(\sP)_0 = \Unif(\{\pm 1\})$ and both denominators are $\frac{1}{2}$.
    The channel overlap is then:
    \begin{align*}
      R_{\sP}(x^{(1)}, x^{(2)})
      &= \Ex_{y \sim \sgn(\sP)_0} (L_{x^{(1)}}^{\sgn}(y) - 1)(L_{x^{(2)}}^{\sgn}(y) - 1) \\
      &= 2\left(\int_{-x^{(1)}}^{\infty} p(y)\,dy\right)\left(\int_{-x^{(2)}}^{\infty} p(y)\,dy\right) + 2\left(\int_{-\infty}^{-x^{(1)}} p(y)\,dy\right)\left(\int_{-\infty}^{-x^{(2)}} p(y)\,dy\right) \numberthis
    \end{align*}
    The Fisher information is then readily computed by differentiating with Leibniz's rule:
    \begin{equation}
        F_{\sP} = 4p(0)^2,
    \end{equation}
    concluding the proof.
\end{proof}

As with censorship, it is plausible that a spectral algorithm is optimal for quantized additive noise models.
This leads to another interesting random matrix theory conjecture, which is again readily verified numerically but seems outside of the reach of standard proof techniques.
\begin{conjecture}
    \label{conj:signed-rmt}
    Let $\bx$ be random as in Assumption~\ref{ass:spike-prior}, $p(y)$ a density as in Proposition~\ref{prop:channel-additive}, and $\sP$ the corresponding additive noise channel.
    Let $\bY^{(0)} \colonequals \frac{\lambda}{\sqrt{n}}\bx\bx^{\top} + \bW$ for $W_{ij} = W_{ji} \sim \sP_0$ i.i.d.\ (i.e., drawn with density $p$), and $W_{ii} = 0$.
    Let $\bY$ have entries $Y_{ij} \colonequals \sgn(Y^{(0)}_{ij})$.
    We conjecture that there exists some $\gamma \in \RR$ depending only on $p$ such that:
    \begin{enumerate}
    \item If $\lambda < 1 / (2p(0))$, then $\frac{1}{\sqrt{n}} \lambda_{\max}(\bY) \to \gamma$ in probability.
    \item If $\lambda > 1 / (2p(0))$, then there is $\epsilon = \epsilon(\lambda) > 0$ such that $\frac{1}{\sqrt{n}} \lambda_{\max}(\bY) \to \gamma + \epsilon$ in probability.
    \end{enumerate}
\end{conjecture}
\noindent
Similar models are treated by the recent work \cite{GKKMZ-2023-NonLinearWignerBBP}, but those results would only apply if the function $\sgn(y)$ in our setting were replaced by a smoother one.

We mention the computations with the calculus of variations that lead to the special densities presented in Corollary~\ref{cor:exp-quantization} as well; we do not give a formal proof, which follows by straightforward calculus.
Instead of maximizing the quantized Fisher information, let us fix it (in other words, fix $p(0)$), and minimize the unquantized Fisher information subject to this constraint.
Suppose then that $p(x)$ is a probability density that minimizes the Fisher information $F[p] = \int_{-\infty}^{\infty} \frac{p^{\prime}(y)^2}{p(y)}dy$ viewed as a functional of $p$, subject to a fixed value of $p(0) = c$.
Then, it must be that the first variation of $F[p]$ is zero, i.e., for any smooth, compactly supported, symmetric $\delta(y)$ with $\delta(0) = 0$, $\int_{-\infty}^{\infty}\delta(y)dy = 0$, and $|\delta(y)| < \epsilon$ for sufficiently small $\epsilon$ (so that adding $\delta$ to $p$ does not change $p(0)$ and leaves $p$ a probability density), we must have that $F[p + \delta] = F[p]$ to leading order.
Expanding this, we have
\begin{align*}
  F[p + \delta]
  &= 2\int_0^{\infty}\frac{(p^{\prime}(y) + \delta^{\prime}(y))^2}{p(y) + \delta(y)}dy \\
  &\approx 2\int_0^{\infty}(p^{\prime}(y)^2 + 2p^{\prime}(y)\delta^{\prime}(y))\left(\frac{1}{p(y)} + \frac{\delta(y)}{p(y)^2}\right)dy \\
  &\approx F[p] + 2\int_0^{\infty}\left(\left(\frac{p^{\prime}(y)}{p(y)}\right)^2 \delta(y) - 2\frac{p^{\prime}(y)}{p(y)}\delta^{\prime}(y)\right)dy. \numberthis
\end{align*}
The only way this can hold for all $\delta$ satisfying our conditions is if $p^{\prime}(y) / p(y)$ is constant, which means that $p(y) = c\exp(-2c|y|)$, a Laplace or two-sided exponential distribution.
The densities featuring in Corollary~\ref{cor:exp-quantization} may be obtained by, say, convolving this density with a Gaussian of variance $\epsilon$ to obtain a sequence of smooth approximations.

Lastly, having performed this derivation, let us retrospectively give some intuition as to why this noise density might be especially resilient to quantization.
Recall from the discussion in Section~\ref{sec:pf:thm:dilution} that in such a non-Gaussian additive model, it is sensible to apply the entrywise transformation of the local Fisher score
\begin{equation}
    f(y) = \frac{1}{F_{\sP}} \frac{\partial}{\partial x}L_x(y) \bigg|_{x = 0} = -\frac{1}{F_{\sP}} \frac{p^{\prime}(y)}{p(y)}
\end{equation}
which approximates a minimum variance unbiased estimator of $x$ from $y$.
But, if we apply this to $p(y) = c\exp(-2c|y|)$, then we get (up to constants) precisely the function $\sgn(y)$ (and if we consider a smooth approximation of such $p(y)$, we will have a smooth approximation of the sign function).
Thus taking the sign of the observations is a transformation we would want to perform anyway on data with this noise distribution, and quantization is relatively benign.

\section*{Acknowledgments}
\addcontentsline{toc}{section}{Acknowledgments}

I thank the participants of the BIRS Workshop on Computational Complexity of Statistical Inference for thoughtful comments following a preliminary presentation of these results, especially Aaron Potechin for pointing out a flaw in an attempted generalization of the results of Appendix~\ref{sec:coord-general}, and Florent Krzakala and Alex Wein for clarifying aspects of the papers \cite{LKZ-2015-LowRankChannelUniversality, PWBM-2018-PCAI}.
I also thank Rayan Saab for bringing the topic of quantization to my attention.

\addcontentsline{toc}{section}{References}
\bibliographystyle{alpha}
\bibliography{main}

\appendix

\section{General tools for coordinate decomposition}
\label{sec:coord-general}

\subsection{Structure of coordinate subspaces}

We first present a useful decomposition of the $V_T$ and $V_{\leq D}$ into orthogonal subspaces and descriptions of orthogonal projection operators onto these subspaces.
We do not need the structure of an LVM, and just work with $\QQ = \sQ_1 \otimes \cdots \otimes \sQ_N$ an arbitrary product measure for $\sQ_i$ probability measures over a measurable space $\Omega$ (it is trivial but notationally cumbersome to allow $\Omega$ to depend on $i$).
The definitions below are an instance of the \emph{Efron-Stein decomposition}, as presented in, e.g., Section~8.3 of \cite{ODonnell-2014-AnalysisBooleanFunctions}.

\begin{definition}[Averaging operator]
    \label{def:Avg}
    For each $T \in [N]$, we write $\Avg_T: L^2(\QQ) \to V_{[N] \setminus T}$ for the operator that maps
    \begin{equation}
        (\Avg_T f)(\by) \colonequals \Ex_{\by_T}[f(\by)] = \EE[f(\by) \mid \by_{[N] \setminus T}].
    \end{equation}
    We write $\Avg_i \colonequals \Avg_{\{i\}}$.
\end{definition}
\noindent
Concretely, the averaging operator for a single coordinate $\Avg_i$ simply ``integrates out'' the $i$th coordinate:
\begin{equation}
    (\Avg_i f)(\by) = \int f(y_1, \dots, y_{i - 1}, \widetilde{y}, y_{i + 1}, \dots, y_N) \, d\sQ_i(\widetilde{y}).
\end{equation}

\begin{proposition}
    \label{prop:V}
    The $V_T$ and $\Avg_T$ satisfy the following properties:
    \begin{enumerate}
    \item $V_S \cap V_T = V_{S \cap T}$.
    \item The $\Avg_i$ commute (i.e., $\Avg_i \Avg_j f = \Avg_j \Avg_i f$ for all $f \in L^2(\QQ)$).
    \item $\Avg_T = \prod_{i \in T} \Avg_i$.
    \item The orthogonal projection operator to $V_T$ is $\Avg_{[N] \setminus T}$.
    \item The subspaces $V_T$ meet orthogonally; that is, for any distinct $S, T \subseteq [N]$, $(V_S \cap V_T)^{\perp} \cap V_S$ is orthogonal to $(V_S \cap V_T)^{\perp} \cap V_T$.
    \end{enumerate}
\end{proposition}
\begin{proof}
    Claims 1 and 3 are immediate by definition, and Claim 2 by Fubini's theorem.
    Claim 4 is a well-known property of conditional expectation (see, e.g., Exercise 4.8 of \cite{Varadhan-2001-ProbabilityTheory}), but may be verified concretely as follows.
    First, clearly from the definition, $\Avg_{[N] \setminus T}$ is idempotent, has $V_T$ as its image, and acts as the identity on $V_T$.
    So, $\Avg_{[N] \setminus T}$ is a projection to $V_T$, and it only suffices to verify that it is orthogonal.
    To do this, we suppose $f \in L^2(\QQ)$ and $g \in V_T$, and compute
    \begin{align*}
      \Ex_{\by \sim \QQ} (f(\by) - \Avg_{[N] \setminus T} f(\by))g(\by)
      &= \Ex_{\by \sim \QQ}f(\by)g(\by) - \Ex_{\by \sim \QQ} \Avg_{[N] \setminus T} f (\by) g(\by) \\
      &= \Ex_{\by \sim \QQ} f(\by) g(\by) - \Ex_{\by \sim \QQ} \EE[f(\by) \mid \by_{T}] \, g(\by) \\
      &= \Ex_{\by \sim \QQ} f(\by) g(\by) - \Ex_{\by \sim \QQ} \EE[f(\by)g(\by) \mid \by_{T}] \\
      &= 0, \numberthis
    \end{align*}
    where we have used that $g(\by)$ depends only on $\by_T$ and the tower property of conditional expectation.
    For Claim~5, note by Claim 1 that $(V_S \cap V_T)^{\perp} \cap V_S = V_{S \cap T}^{\perp} \cap V_S$, and by Claim~4 the orthogonal projection to $V_{S \cap T}^{\perp}$ is $\Id - \Avg_{[N] \setminus (S \cap T)}$ and the orthogonal projection to $V_S$ is $\Avg_{[N] \setminus S}$.
    In particular, these projections commute, so the orthogonal projection to $(V_S \cap V_T)^{\perp} \cap V_S$ is their product, $\Avg_{[N] \setminus S}(\Id - \Avg_{[N] \setminus (S \cap T)}) = \Avg_{[N] \setminus S} - \Avg_{[N] \setminus (S \cap T)}$.
    Similarly, the orthogonal projection to $(V_S \cap V_T)^{\perp} \cap V_T$ is $\Avg_{[N] \setminus S} - \Avg_{[N] \setminus (S \cap T)}$.
    Finally, it suffices to check that the product of these projections is zero:
    \begin{equation}
        (\Avg_{[N] \setminus S} - \Avg_{[N] \setminus (S \cap T)})(\Avg_{[N] \setminus T} - \Avg_{[N] \setminus (S \cap T)}) = 0,
    \end{equation}
    since each of the four terms upon expanding is $\Avg_{[N] \setminus (S \cap T)}$, occurring twice positively and twice negatively.
\end{proof}

We now define a subtler family of subspaces, which aim to capture the ``new'' functions depending only on $\by_T$ that $V_T$ contains, beyond those that are linear combinations of ones depending on smaller coordinate subsets.
\begin{definition}
    For each $T \subseteq [N]$, $\what{V}_T \colonequals V_T \cap (\sum_{S \subsetneq T} V_S)^{\perp}$.
\end{definition}
\noindent
We establish the following important properties of these subspaces.

\begin{proposition}
    \label{prop:Vhat}
    The $\what{V}_T$ satisfy the following properties:
    \begin{enumerate}
    \item The subspaces $\what{V}_T$ are mutually orthogonal; that is, for any distinct $S, T \subseteq [N]$ and $f \in \what{V}_S$ and $g \in \what{V}_T$, we have $\EE_{\by \sim \QQ} f(\by)g(\by) = 0$.
    \item $V_T = \bigoplus_{S \subseteq T} \what{V}_S$; in particular, $L^2(\QQ) = V_{[N]} = \bigoplus_{T \subseteq [N]} \what{V}_T$.
    \item The orthogonal projection operator to $\what{V}_T$ is
        \begin{equation}
            \prod_{i \in T} (\Id - \Avg_i) \prod_{i \in [N] \setminus T} \Avg_i = \sum_{S \subseteq T} (-1)^{|T| - |S|} \Avg_{[N] \setminus S}
            \label{eq:Vhat-proj}
        \end{equation}
    \end{enumerate}
\end{proposition}
\begin{proof}
    For Claim 1, note that $\what{V}_S \subseteq V_{S \cap T}^{\perp} \cap V_S$ and $\what{V}_T \subseteq V_{S \cap T}^{\perp} \cap V_T$.
    The result then follows by Claims 1 and 5 of Proposition~\ref{prop:V}.

    For Claim 2, we may proceed by induction on $|T|$.
    The result holds for $T = \emptyset$, in which case $V_T = \what{V}_T$ is the span of the constant function.
    Now, suppose the result holds for all $|T| \leq k - 1$, and we have $|T| = k$.
    For any $T^{\prime} \subsetneq T$, we have by the inductive hypothesis $\bigoplus_{S \subseteq T} \what{V}_S \supseteq \bigoplus_{S \subseteq T^{\prime}} \what{V}_S = V_{T^{\prime}}$.
    In particular, $\bigoplus_{S \subseteq T} \what{V}_S \supseteq \sum_{S \subsetneq T} V_S$.
    On the other hand, we also have $\bigoplus_{S \subseteq T} \what{V}_S \supseteq \what{V}_T = V_T \cap (\sum_{S \subsetneq T} V_S)^{\perp}$.
    Thus, $\bigoplus_{S \subseteq T} \what{V}_S \supseteq \sum_{S \subsetneq T} V_S + V_T \cap (\sum_{S \subsetneq T} V_S)^{\perp} \supseteq V_T$.
    The opposite inclusion is immediate, completing the induction.

    For Claim 3, note that, by Claim 2, writing $\what{P}_T$ for the orthogonal projection to $\what{V}_T$, we have $\sum_{S \subseteq T} \what{P}_S = \Avg_{[N] \setminus T}$.
    The result, in the form of the right-hand side of \eqref{eq:Vhat-proj}, then follows by the \Mobius\ inversion formula.
\end{proof}

It may be instructive to identify the manifestations of these somewhat abstract objects in Boolean function theory.
\begin{example}[Boolean Fourier analysis]
    \label{ex:boolean-fourier}
    Consider the case $\Omega = \{\pm 1\}$ and $\QQ = \Unif(\{\pm 1\}^N)$.
    Then, $L^2(\QQ)$ is simply the set of all functions $f: \{\pm 1\}^N \to \RR$.
    As is well-known, every such function may be written as a polynomial of multilinear monomials,
    \begin{equation}
        f(\by) = \sum_{T \subseteq [N]} \what{f}(T) \by^T,
    \end{equation}
    also known as the Boolean Fourier expansion of $f$, where $\by^T = \prod_{i \in T}y_i$.
    The monomials $\by^T$ form an orthonormal basis for $L^2(\QQ)$.

    In this representation, $V_T$ is the span of all $\by^S$ for $S \subseteq T$.
    As a consequence, the degree of $f$ (as a polynomial in the Fourier representation) equals the coordinate degree of $f$.
    The operator $\Avg_i$ acts on $f$ by removing all monomials including $y_i$ from this representation,
    \begin{equation}
        (\Avg_if)(\by) = \sum_{T \subseteq [N] \setminus \{i\}} \what{f}(T) \by^T,
    \end{equation}
    so indeed $\prod_{i \in [N] \setminus T} \Avg_i$ is the orthogonal projection to $V_T$.
    The commutativity of these $\Avg_i$ is also clear.

    $\what{V}_T$ consists of the functions whose monomials only have indices lying in $T$ (i.e., belonging to $V_T$) but also having no monomials whose indices are contained in any proper subset of $T$ (i.e., orthogonal to $\sum_{S \subsetneq T} V_S$).
    Therefore, $\what{V}_T$ is merely the span of the single monomial function $\by^T$.
    We may equate this with the projection formula \eqref{eq:Vhat-proj} by noting that $\Id - \Avg_i$ acts by removing all monomials not including $y_i$ from the Fourier representation, and thus the only monomial that the projection does not remove is $\by^T$.

    Similar results equating coordinate with polynomial degree and Efron-Stein decomposition with Boolean Fourier decomposition hold also for a domain $\Omega^N$ with $\Omega$ a finite set, where polynomials may be made sense of through a ``one-hot'' encoding of $\bx \in \Omega^N$.
\end{example}
\noindent
We thus see that the orthogonal decomposition $L^2(\QQ) = \bigoplus_{T \subseteq [N]} \what{V}_T$ may be viewed as a generalization of Boolean Fourier analysis to arbitrary multivariate $L^2$ spaces coming from a product measure.

Finally, this machinery gives us a simple way to project arbitrary functions to $V_{\leq D}$.
The condensed formula below will not prove very useful to us---we will prefer to view the projection to $V_{\leq D}$ as a sum of projections to the orthogonal components $\what{V}_T$---but we emphasize that it gives a fully explicit description of this projection as a linear combination of averaging operators.
\begin{lemma}
    \label{lem:PVD}
    The orthogonal projection operator to $V_{\leq D}$ is
    \begin{equation}
        P_{\leq D} = \sum_{\substack{T \subseteq [N] \\ |T| \leq D}} (-1)^{D - |T|} \binom{N - |T| - 1}{D - |T|} \Avg_{[N] \setminus T}.
    \end{equation}\
\end{lemma}
\begin{proof}
    We use that, per Proposition~\ref{prop:Vhat}, $V_{\leq D}$ is the direct sum of orthogonal components $\what{V}_T$ for $|T| \leq D$.
    Thus, $P_{\leq D}$ is the sum of the corresponding projections:
    \begin{align*}
      P_{\leq D}
      &= \sum_{\substack{T \subseteq [N] \\ |T| \leq D}} \what{P}_T \\
  &= \sum_{\substack{T \subseteq [N] \\ |T| \leq D}} \Avg_{[N] \setminus T} \,\, \sum_{S \subseteq T} (-1)^{|S|} \Avg_S \\
  &= \sum_{\substack{T \subseteq [N] \\ |T| \leq D}} \sum_{S \subseteq T} (-1)^{|S|} \Avg_{[N] \setminus (T \setminus S)}
  \intertext{where, introducing $R \colonequals T \setminus S$,}
  &= \sum_{\substack{R \subseteq [N] \\ |R| \leq D}} \left(\sum_{\substack{R \subseteq T \subseteq [N] \\ |T| \leq D}} (-1)^{|T| - |R|}\right)\Avg_{[N] \setminus R} \\
  &= \sum_{\substack{R \subseteq [N] \\ |R| \leq D}} \left(\sum_{t = |R|}^{D} (-1)^{t - |R|}\binom{N - |R|}{t - |R|} \right)\Avg_{[N] \setminus R}, \numberthis
\end{align*}
and the final formula follows from a combinatorial calculation.
\end{proof}

\subsection{Alternative derivation with concrete bases}

The above framework may also be established more concretely using the bases of orthogonal polynomials usually used in the analysis of LDP.
More generally, for each $i \in [N]$, let $\{f_{i, j}\}_{j \geq 0}$ be a countable basis of (univariate) orthonormal functions in $L^2(\sQ_i)$, with $f_{i, 0} = 1$.
These then form a countable orthonormal product basis $\{f_{\bm j}\}_{\bm j \in \NN^N}$ of $L^2(\QQ)$ defined by $f_{\bm j}(\by) = \prod_{i = 1}^N f_{i, j_i}(y_i)$.

Then, $V_T$ is the span of those $f_{\bm j}$ with $j_i = 0$ for all $i \notin T$.
And, $\what{V}_T$ is the span of those $f_{\bm j}$ with $j_i = 0$ for all $i \not \in T$ and $j_i > 0$ for all $i \in T$ (i.e., those $\bm j$ whose support is exactly $T$).
The basis elements spanning the $\what{V}_T$ then partition the basis, which immediately gives their orthogonality and the decomposition $L^2(\QQ) = \bigoplus_{T \subseteq [N]} \what{V}_T$.

On the other hand, our formulation above makes it clear that this decomposition is independent of the choice of orthonormal bases, which is the entire benefit of working with coordinate degree---that we can do calculations without needing to understand the details of the underlying orthogonal polynomials or orthonormal functions.

\subsection{Coordinate decomposition of likelihood ratio}

Suppose now that we also have a further probability measure $\PP$ on $\Omega^N$.
We do not make the assumptions of an LVM, and rather allow $\PP$ to be arbitrary so long as it is absolutely continuous to $\QQ$ and has $d\PP / d\QQ \in L^2(\QQ)$ (as for a good LVM in the main text).
In the main text, we are interested in the decomposition over the subspaces $\what{V}_T$ of the likelihood ratio $L \colonequals d\PP / d\QQ$.

The projections of $L$ to $V_T$, per Proposition~\ref{prop:V}, are
\begin{equation}
    L_T \colonequals \Ex_{\by_{[N] \setminus T}} L(\by).
\end{equation}
One may check that this is the natural definition of the likelihood ratio between the \emph{marginal} law $\PP_T$ of $\PP$ on the coordinates of $T$ and the corresponding marginal law $\QQ_T = \bigotimes_{i \in T} \sQ_i$, since
\begin{equation}
    \Ex_{\by_T \sim \QQ_T} L_T(\by_T) f(\by_T) = \Ex_{\by \sim \QQ} L(\by) f(\by_T) = \Ex_{\by \sim \PP} f(\by_T).
\end{equation}
That is, abstractly one would use this to \emph{define} the marginal law $\PP_T$ by its Radon-Nikodym derivative to $\QQ_T$.
But, in simple cases of discrete measures or continuous measures with well-behaved densities, one may also check that $L_T$ is indeed the ratio of marginal probability mass or density functions, respectively.
We thus refer to the $L_T$ as \emph{marginal likelihood ratios}.

From the previous results we then immediately obtain the following.
\begin{lemma}
    \label{lem:lr-decomp}
    Define the LCDLR, $L_{\leq D} \colonequals P_{\leq D} L$, and
    \begin{equation}
        \what{L}_T \colonequals \sum_{S \subseteq T} (-1)^{|T| - |S|} L_S.
    \end{equation}
    Then, $\what{L}_T$ is the orthogonal projection of $L$ to $\what{V}_T$.
    In particular, the $\what{L}_T$ are orthogonal, and
    \begin{equation}
        \Ex_{\by \sim \QQ} L_{\leq D}(\by)^2 = \sum_{\substack{T \subseteq [N] \\ |T| \leq D}} \Ex_{\by \sim \QQ} \what{L}_T(\by)^2.
    \end{equation}
\end{lemma}

\subsection{Additional remarks}
\label{sec:coord-remarks}

We point out a few more curious features and consequences of our observations that are less essential to our arguments in the main text.

\paragraph{Expanding the LCDLR}
Note that both the norm of the LCDLR and the LCDLR itself may be written completely in terms of the marginal likelihood ratios $L_T$, upon expanding in Lemma~\ref{lem:lr-decomp}.
For the LCDLR, we have either by expanding in the Lemma or by using the formula from Lemma~\ref{lem:PVD},
\begin{equation}
    L_{\leq D} = \sum_{\substack{T \subseteq [N] \\ |T| \leq D}} (-1)^{D - |T|} \binom{N - |T| - 1}{D - |T|} L_T. \label{eq:lcdlr-app}
\end{equation}
\noindent
We do not directly use the formula \eqref{eq:lcdlr-app}, but we include it to emphasize that the LCDLR takes a surprisingly explicit form in arbitrary models so long as $\QQ$ is a product measure.
We leave it as an interesting question to ascertain whether the $\what{L}_T$ themselves have a natural statistical meaning for hypothesis testing or whether \eqref{eq:lcdlr-app} can be operationalized into a useful algorithm.

For the norm, one may expand in Lemma~\ref{lem:lr-decomp} into inner products of the $L_S$, which have the following appealing structure reflecting the lattice structure of subsets of $[N]$, analogous to Claim~1 of Proposition~\ref{prop:V} that $V_S \cap V_T = V_{S \cap T}$.
\begin{proposition}
    $\Ex_{\by \sim \QQ} L_S(\by) L_T(\by) = \Ex_{\by \sim \QQ} L_{S \cap T}(\by)^2$.
\end{proposition}
\noindent
This follows from the decomposition, by Proposition~\ref{prop:Vhat}, $L_T = \sum_{R \subseteq T} \what{L}_R$, where the $\what{L}_R$ are mutually orthogonal.
Alternatively, it may be viewed as a consequence of $V_S \cap V_T = V_{S \cap T}$ and the orthogonal meeting of these subspaces from Proposition~\ref{prop:V}.

\paragraph{Connections with $\chi^2$ divergence}
As a special case of Lemma~\ref{lem:lr-decomp} on the norm of the LCDLR, we may consider the norm of $L$ itself, which is related to the $\chi^2$ divergence between $\PP$ and $\QQ$.
The Lemma thus gives an interesting decomposition of this divergence.
\begin{definition}
    \label{def:chi-squared}
    The \emph{$\chi^2$ divergence} between $\PP$ and $\QQ$ is
    \begin{equation}
        \chi^2(\PP \dbar \QQ) \colonequals \Ex_{\by \sim \QQ}L(\by)^2 - 1 = \Ex_{\by \sim \QQ}\left(L(\by) - 1\right)^2.
        \label{eq:chi-squared}
    \end{equation}
\end{definition}

\begin{corollary}
    \label{cor:chi-squared}
    For any $\PP$ and $\QQ = \sQ_1 \otimes \cdots \otimes \sQ_N$,
    \begin{equation}
        \chi^2(\PP \dbar \QQ) = \sum_{\substack{T \subseteq [N] \\ T \neq \emptyset}} \Ex_{\by \sim \QQ}\left(\sum_{S \subseteq T} (-1)^{|T| - |S|}L_S(\by)\right)^2.
    \end{equation}
\end{corollary}
\begin{proof}
    The result follows from the definition of $\chi^2$ divergence, Lemma~\ref{lem:lr-decomp}, and the observation that $L_{\emptyset} = 1$.
\end{proof}
\noindent
In fact, the same kind of identity also holds for an arbitrary partition of $[N]$, since we may group the $y_1, \dots, y_N$ into various subsets of ``effective coordinates'' of $\by$ and repeat the same analysis.
Actually, the definition \eqref{eq:chi-squared} itself is just the special case of these formulas when we view $\by$ as having just one coordinate, in which case there is just one term in the summation above that is equal to $\EE_{\by \sim \QQ}(L_{\{1\}}(\by) - L_{\emptyset}(\by))^2 = \EE(L(\by) - 1)^2 = \EE L(\by)^2 - 1$.
Generally, when computing a $\chi^2$ divergence, if any subsets of coordinates of $\QQ$ are independent, then we may apply such a ``coarsening'' operation and apply the identity to obtain an interesting decomposition.

In a similar spirit, it is tempting to bound the components of the norm of the LCDLR by $\chi^2$ divergences,
\begin{equation}
    \Ex_{\by \sim \QQ} \what{L}_T(\by)^2 \leq \Ex_{\by \sim \QQ} (L_T(\by) - 1)^2 = \chi^2(\PP_T \dbar \QQ_T).
\end{equation}
This is valid for all $T \neq \emptyset$, since one may view $\what{L}_T$ as $L_T$ orthogonalized against all $\what{L}_S$ with $S \subsetneq T$, while $L_T - 1$ is $L_T$ orthogonalized against just $\what{L}_{\emptyset} = 1$.
We could thus bound the norm of the LCDLR by a sum of marginal $\chi^2$ divergences,
\begin{equation}
    \Ex_{\by \sim \QQ} L_{\leq D}(\by)^2 \leq 1 + \sum_{\substack{T \subseteq [N] \\ |T| \leq D}} \chi^2(\PP_T \dbar \QQ_T),
\end{equation}
and in fact one may bound more carefully by a sum over just the $|T| = D$ as well.
Unfortunately, while conceptually appealing, this bound may be verified to be too loose even in, e.g., the spiked matrix models treated in the main text.
The issue is that, upon expanding these divergences as sums of norms of the $\what{L}_S$ per Corollary~\ref{cor:chi-squared}, we find that the norms of certain $\what{L}_S$, say those with $|S| = D / 2$, are dramatically overcounted in this bound.

\paragraph{Beyond product measures?}
Finally, we emphasize that, unfortunately, our approach no longer applies once $\QQ$ is not a product measure (or, less restrictively per the above, admits no non-trivial partition into subsets of independent indices).
The reason for this is subtle; it is, after all, still true that $\Avg_{[N] \setminus T}$ (with the interpretation as a conditional expectation) is an orthogonal projection to $V_T$; the proof of this in Proposition~\ref{prop:V} still holds.
The problem is that $\Avg_S$ no longer factorizes in terms of the $\Avg_i$, nor do these projections commute, and therefore nor do the $V_T$ meet orthogonally.
Indeed, consider the case $N = 2$ where $y_1$ and $y_2$ are correlated.
Intuitively speaking, $\Avg_2 y_2 \in V_{\{1\}}$ is the ``best prediction'' of $y_2$ as a function of $y_1$.
$\Avg_1\Avg_2 y_2 \in V_{\{2\}}$ is then the best prediction of this function of $y_1$ as a function of $y_2$---but in the presence of a non-trivial correlation, this is not a constant, while $\Avg_{\{1, 2\}} y_2 = \EE y_2$ is.

On the other hand, the right-hand side of \eqref{eq:lcdlr-app} is still sensible in arbitrary models, as a kind of ``mean-field approximation'' of the LCDLR, in the sense of computing what the LCDLR would be if $\QQ$ were a product measure.
We offer the intriguing question of whether this quantity may be useful for efficient hypothesis testing even without the product structure of $\QQ$.

\section{Necessity of conditions for universality}
\label{sec:non-universality}

We give some elaboration on the conditions that are necessary for universality and when universality of computational thresholds for strong detection and other tasks should break down.

The following is a very simple example showing that \emph{some} control of the prior is necessary for universality of the coordinate advantage to hold.

\begin{example}[Non-universality]
    \label{ex:non-universality}
    Consider the prior $\sX = \Unif(\{\pm \frac{1}{2}\}^N)$ observed through the Bernoulli and Gaussian channels.
    Through $\sP_x = \Ber(\frac{1}{2} + x)$, when $x \in \{\pm \frac{1}{2}\}$ we have $\sP_x = \delta_x$, so the observation is just $\PP = \Ber(\frac{1}{2})^{\otimes N} = \QQ$.
    Therefore,
    \begin{equation}
        \CAdv_{\leq D}\left(\sX, \Ber\left(\frac{1}{2} + x\right)\right) = 1
    \end{equation}
    for all $D \geq 0$.
    On the other hand, through $\sP_x = \sN(x, 1)$, for any $D \geq 2$,
    \begin{align*}
      \CAdv_{\leq D}(\sX, \sN(x, 1))
      &\geq \Adv_{\leq D}(\sX, \sN(x, 1)) \\
      &= \Ex_{\bx^{(1)}, \bx^{(2)} \sim \sX} \exp^{\leq D}(\langle \bx^{(1)}, \bx^{(2)} \rangle) \\
      &\geq \EE \frac{1}{2}\langle \bx^{(1)}, \bx^{(2)}\rangle^2 \\
      &= \frac{1}{8}N, \numberthis
    \end{align*}
    which diverges as $N \to \infty$ (in the inequality we appeal to Lemma~\ref{lem:non-neg}).
\end{example}
\noindent
This example illustrates that some assumption on the size of the prior and its entries is necessary, and one may check that this example fails Assumption P2 of Theorem~\ref{thm:channel-universality}.
Generally, our brand of universality holds when the dependence on the channel $\sP_x$ is only on its \emph{local} behavior near $x = 0$; priors whose entries are typically very far from zero like this one can bring out the \emph{global} differences of channels.
Here, the difference is stark: the Bernoulli channel for a certain magnitude of signal ``tells the truth'' and outputs the signal with no noise at all, while the Gaussian channel always applies some non-trivial noise to the signal.

\begin{remark}[Detection versus recovery]
    As an aside, this example also happens to showcase the subtle relationship between detection (hypothesis testing) and recovery (estimation) of $\bx$.
    The above shows that detection is much easier through the Gaussian channel than through the Bernoulli channel.
    Conversely, perfect recovery of $\bx$ from $\by \sim \PP$ is trivially possible through the Bernoulli channel, since $\by = \bx$.
    On the other hand, through the Gaussian channel we only make noisy observations of $\bx \in \{\pm \frac{1}{2}\}^N$, and a simple argument shows that it is only possible to estimate a constant fraction of the $x_i$ correctly with high probability.
\end{remark}

\paragraph{Sparse PCA}
On the other hand, there are situations that Theorem~\ref{thm:channel-universality} does not cover where we still expect universality to hold.
One is \emph{sparse PCA}.
As a specific example, studied for LDP by \cite{DKWB-2019-SubexponentialTimeSparsePCA}, consider the spiked matrix model with the sparse Rademacher prior from Section~\ref{sec:app}, but with decaying sparsity $s = s(n) = o(1)$.
Recall that the prior is the law of $\bX = \frac{\lambda}{\sqrt{n}}\bx\bx^{\top}$ where $\bx$ has $n$ i.i.d.\ entries drawn from the measure
\begin{equation}
    \pi = (1 - s)\delta_0 + \frac{s}{2}\delta_{1 / \sqrt{s}} + \frac{s}{2}\delta_{-1 / \sqrt{s}},
\end{equation}
and the total dimension is $N = n^2$.
(Here $\lambda$ is still a constant not depending on $n$; we are also ignoring the symmetry of the observed matrices for the sake of simplicity, which does not make a material difference.)
In particular, with high probability we have that the norms of the vectorization of the signal are
\begin{equation}
    \|\mathsf{vec}(\bX)\|_k \asymp n^{-1/2} \left((sn)^2 \cdot s^{-k}\right)^{\frac{1}{k}} = s^{\frac{2}{k} - 1} n^{\frac{2}{k} - \frac{1}{2}} = s^{\frac{2}{k} - 1} N^{\frac{1}{k} - \frac{1}{4}}.
\end{equation}
Thus, Assumptions P2 and P3 (for Theorem~\ref{thm:channel-universality}) concerning $k > 2$ are not satisfied once $s = o(1)$.
However, looking more deeply into the proof of Theorem`\ref{thm:channel-universality}, we see that the true source of non-universality would be for the quantity $\sum_{i, j = 1}^n (X^{(1)}_{ij})^2 (X^{(2)}_{ij})^2$ to be unbounded for typical independent draws $\bX^{(1)}, \bX^{(2)}$ from the prior.
In the proof, we bound this by the Cauchy-Schwarz inequality, which does not respect the sparsity structure in the present setting.
Indeed, we expect to have with high probability
\begin{equation}
    \sum_{i, j = 1}^n (X^{(1)}_{ij})^2 (X^{(2)}_{ij})^2 \asymp \frac{1}{n^2} \left(\sum_{i = 1}^n (x_i^{(1)})^2(x_i^{(2)})^2\right)^2 \asymp \frac{1}{n^2} \left(s^2 n \cdot \frac{1}{s^2}\right)^2 = 1.
\end{equation}
The result of Theorem~\ref{thm:channel-universality} should therefore likely hold for sparse PCA as well, but would require a more tailored proof treating sparsity more carefully.

\paragraph{Constants and weak detection rates}
Our proof of Theorem~\ref{thm:channel-universality} makes it apparent that the constants $C_1, C_2, C_3$ are generally necessary and depend on the channel $\sP$ in a way not captured by the Fisher information.
(Indeed, they arise precisely from the error term discussed above, which appears in the fourth order terms of the Taylor expansion of $R_{\sP}$, while the Fisher information only governs the second order term.)
In particular, these constants cannot in general be replaced by $1 + o(1)$.
As discussed in, e.g., \cite{BAHSWZ-2022-FranzParisiLowDegree,RSWY-2022-CountCommunitiesLowDegree}, showing that the coordinate advantage is $1 + o(1)$ may be viewed as evidence that \emph{weak detection}---distinguishing $\PP_n$ and $\QQ_n$ with some probability strictly greater than $\frac{1}{2}$---is computationally hard.
Thus we expect that thresholds for weak detection should not enjoy the channel universality of thresholds for strong detection.
The remarkable recent work of \cite{MW-2023-PreciseErrorRatesPCA} made this connection more precise, showing that the numerical value of the advantage is related (at least in a spiked matrix model) to the best receiver operating characteristic (ROC) curve, or tradeoff between Type~I and Type~II errors, achievable by efficient testing algorithms.
We therefore expect this optimal tradeoff, too, to be non-universal.
We propose the verification of this as an interesting open problem.

\section{Omitted proofs}
\subsection{Assumptions for additive channels: Proof of Proposition~\ref{prop:channel-additive}}
\label{sec:pf:prop:channel-additive}

Recall that we are considering a channel with $\sP_x$ the law of $x + z$ for $z \sim \rho$, where $\rho$ has a density $p(y)$ on $\RR$.
This means that $\sP_x$ has a density $p(y - x)$, and the channel likelihood ratios are
\begin{equation}
    L_x(y) = \frac{d\sP_x}{d\sP_0} = \frac{p(y - x)}{p(y)},
\end{equation}
which is always defined since we have assumed that $p(y) > 0$ for all $y$.
The channel overlap is then
\begin{equation}
    R_{\sP}(x^{(1)}, x^{(2)}) = -1 + \Ex_{y \sim \rho} L_{x^{(1)}}(y) L_{x^{(2)}}(y) = -1 + \int_{-\infty}^{\infty} \frac{p(y - x^{(1)})p(y - x^{(2)})}{p(y)} dy.
\end{equation}
This is, up to the $-1$ term and a logarithm, equivalent to the \emph{translation function} considered by \cite{PWBM-2018-PCAI} in the special context of PCA.
To reiterate an example mentioned there:
\begin{example}
    When $\rho = \sN(0, 1)$, then $K(x^{(1)}, x^{(2)}) = \exp(x^{(1)}x^{(2)}) - 1$.
\end{example}

\begin{proof}[Proof of Proposition~\ref{prop:channel-additive}]
It suffices to verify that the assumptions of Theorem~\ref{thm:channel-universality} hold.
Since we assume $p(y) > 0$ and $p$ is $\sC^4$ on all of $\RR$, $L_x(y)$ is also $\sC^4$ in $x$ on all of $\RR$, so Assumption C1 of Theorem~\ref{thm:channel-universality} is satisfied.
For Assumption C2, we may compute
\begin{equation}
    \frac{\partial^3 R_{\sP}}{\partial x^{(1)^2} \partial x^{(2)}}(0, 0) = -\int_{-\infty}^{\infty} \frac{p^{\prime\prime}(y)p^{\prime}(y)}{p(y)} dy = 0,
\end{equation}
since $p(-y) = p(y)$, so $p^{\prime\prime}(-y) = p^{\prime\prime}(y)$ while $p^{\prime}(-y) = -p^{\prime}(y)$, whereby the integrand is an odd function.
The Fisher information is computed similarly:
\begin{equation}
    F_{\sP} = \frac{\partial^2 R_{\sP}}{\partial x^{(1)} \partial x^{(2)}}(0, 0) = \int_{-\infty}^{\infty} \frac{p^{\prime}(y)^2}{p(y)}dy,
\end{equation}
as claimed.
\end{proof}

\begin{remark}
    The role the translation function plays in our argument is somewhat different from that in \cite{PWBM-2018-PCAI}.
    The reason for this is that their calculations (which concern the full $\chi^2$ divergence or norm of the likelihood ratio) may be viewed as taking $D = N$ in the formula (not the bound) of Theorem~\ref{thm:lvm}, factorizing the resulting expression into a product, and writing this as an exponential of a sum.
    More specifically, abbreviating $R_i \colonequals R_{\sP}(x_i^{(1)}, x_i^{(2)})$,
    \begin{align*}
      \CAdv_{\leq N}(\sX, \sP)^2
      &= \Ex_{\bx^{(1)}, \bx^{(2)}\sim \sX} \sum_{T \subseteq [N]} \prod_{i \in T} R_{i} \\
      &= \Ex_{\bx^{(1)}, \bx^{(2)}} \prod_{i = 1}^N(1 + R_i) \\
      &= \Ex_{\bx^{(1)}, \bx^{(2)}} \exp\left(\sum_{i = 1}^N \log(1 + R_i)\right), \numberthis
    \end{align*}
    and this is the source of the extra logarithm in the translation function they work with.
    We obtain a bound that instead has $\sum_{i = 1}^N R_i$ in the exponential, which is larger, but our approach has the benefit of being easily adaptable to low coordinate degree truncations.
\end{remark}

\subsection{Assumptions for exponential family channels: Proof of Proposition~\ref{prop:channel-exponential}}
\label{sec:pf:prop:channel-exponential}

Let us first review some basic definitions and properties of exponential families that we omitted in the presentation of the main results in the Introduction.
Our discussion is similar to that of the same topics in \cite{Kunisky-2020-LowDegreeMorris}.

\begin{definition}
    \label{def:exp-family}
    Let $\widetilde{\sP}_0$ be a probability measure over $\RR$ which is not a single atom.
    Let $\psi(\theta) \colonequals \log \EE_{x \sim \widetilde{\PP}_0}[\exp(\theta x) ]$ and $\Theta \colonequals \{\theta \in \RR: \psi(\theta) < \infty\}$.
    Then, the \emph{natural exponential family (NEF) generated by $\widetilde{\sP}_0$} is the family of probability measures $\rho_{\theta}$, for $\theta \in \Theta$, given by the relative densities
    \begin{equation}
        \frac{d\widetilde{\sP}_\theta}{d\widetilde{\sP}_0}(y) \colonequals \exp(\theta y - \psi(\theta)).
    \end{equation}
\end{definition}

$\psi(\theta)$ is the cumulant generating function of $\widetilde{\sP}_0$.
The cumulant generating functions of the $\widetilde{\sP}_{\theta}$ are translations of $\psi$, and the means and variances of $\widetilde{\sP}_{\theta}$ are thus given by derivatives of $\psi$ evaluated away from zero,
\begin{align}
  \mu_\theta &\colonequals \EE_{y \sim \widetilde{\sP}_\theta}[y] = \psi^\prime(\theta), \\
  \sigma_\theta^2 &\colonequals \Var_{y \sim \widetilde{\sP}_\theta}[y] = \psi^{\prime\prime}(\theta).
\end{align}
For compatibility with the statement of Proposition~\ref{prop:channel-exponential}, we also write
\begin{align}
  \mu &\colonequals \mu_0, \\
  \sigma^2 & \colonequals \sigma^2_0.
\end{align}

Since $\widetilde{\sP}_0$ is not an atom, neither is any $\widetilde{\sP}_{\theta}$, so $\psi^{\prime\prime}(\theta) = \sigma_{\theta}^2 > 0$.
Therefore, $\psi^\prime$ is strictly increasing, and thus one-to-one.
Letting $\Sigma \colonequals \psi^{\prime}(\RR) - \mu \subseteq \RR$, which one may verify is always some open interval, possibly infinite on either side, of $\RR$, we see that $\widetilde{\sP}_\theta$ admits an alternative parametrization by the mean, which we may also translate by $\mu \colonequals \mu_0$ to obtain the setting specified in Proposition~\ref{prop:channel-exponential},
\begin{align}
  \theta(x) &\colonequals \psi^{\prime^{-1}}(\mu + x), \\
  \sP_{x} &\colonequals \widetilde{\sP}_{\theta(x)}.
\end{align}
Note that $\theta(0) = 0$, so $\sP_{0} = \widetilde{\sP}_0$.

The channel likelihood ratios in this parametrization are then
\begin{equation}
    L_x(y) = \frac{d\sP_x}{d\sP_0}(y) = \frac{d\widetilde{\PP}_{\theta(x)}}{d\widetilde{\PP}_0}(y) = \exp\big(\theta(x) y - \psi(\theta(x))\big),
\end{equation}
and the channel overlap is
\begin{align*}
  R_{\sP}(x^{(1)}, x^{(2)})
  &= \Ex_{y \sim \sP_0} L_{x^{(1)}}(y) L_{x^{(2)}}(y) - 1 \\
  &= \exp(-\psi(\theta(x^{(1)})) - \psi(\theta(x^{(2)}))) \Ex_{y \sim \sP_0} \exp(y(\theta(x^{(1)}) + \theta(x^2))) - 1 \\
  &= \exp\bigg(\psi\big(\theta(x^{(1)}) + \theta(x^{(2)})\big) - \psi(\theta(x^{(1)})) - \psi(\theta(x^{(2)}))\bigg) - 1.
\end{align*}
It will also be useful to define
\begin{equation}
    K_{\sP}(x^{(1)}, x^{(2)}) \colonequals \exp\bigg(\psi\big(\theta(x^{(1)}) + \theta(x^{(2)})\big) - \psi(\theta(x^{(1)})) - \psi(\theta(x^{(2)}))\bigg) = R_{\sP}(x^{(1)}, x^{(2)}) + 1.
\end{equation}

\begin{proof}[Proof of Proposition~\ref{prop:channel-exponential}]
    Again we verify the conditions of Theorem~\ref{thm:channel-universality}, which requires taking the first three partial derivatives of $R_{\sP}$.
    For the first derivative:
    \begin{equation}
        \frac{\partial R_{\sP}}{\partial x^{(1)}}(x^{(1)}, x^{(2)}) = K_{\sP}(x^{(1)}, x^{(2)}) \theta^{\prime}(x^{(1)}) \left(\psi^{\prime}(\theta(x^{(1)}) + \theta(x^{(2)})) - \psi^{\prime}(\theta(x^{(1)}))\right).
    \end{equation}
    For the mixed second partial derivative:
    \begin{align*}
      &\frac{\partial^2 R_{\sP}}{\partial x^{(1)}\partial x^{(2)}}(x^{(1)}, x^{(2)}) = K_{\sP}(x^{(1)}, x^{(2)}) \theta^{\prime}(x^{(1)})\theta^{\prime}(x^{(2)}) \\
      &\hspace{1cm} \bigg(\left(\psi^{\prime}(\theta(x^{(1)}) + \theta(x^{(2)})) - \psi^{\prime}(\theta(x^{(1)}))\right)\left(\psi^{\prime}(\theta(x^{(1)}) + \theta(x^{(2)})) - \psi^{\prime}(\theta(x^{(2)}))\right) \\
      &\hspace{1.5cm} + \psi^{\prime\prime}(\theta(x^{(1)}) + \theta(x^{(2)}))\bigg). \numberthis
    \end{align*}
    Finally, for the mixed third partial derivative:
    \begin{align*}
      &\frac{\partial^3 R_{\sP}}{\partial x^{(1)^2}\partial x^{(2)}}(x^{(1)}, x^{(2)}) = K_{\sP}(x^{(1)}, x^{(2)}) \theta^{\prime}(x^{(2)}) \\
      &\hspace{1cm} \bigg(\theta^{\prime}(x^{(1)})^2 \bigg(\left(\psi^{\prime}(\theta(x^{(1)}) + \theta(x^{(2)})) - \psi^{\prime}(\theta(x^{(1)}))\right)\left(\psi^{\prime}(\theta(x^{(1)}) + \theta(x^{(2)})) - \psi^{\prime}(\theta(x^{(2)}))\right) \\
      &\hspace{3cm} + \psi^{\prime\prime}(\theta(x^{(1)}) + \theta(x^{(2)}))\bigg)\left(\psi^{\prime}(\theta(x^{(1)}) + \theta(x^{(2)})) - \psi^{\prime}(\theta(x^{(1)}))\right) \\
      &\hspace{1cm} + \theta^{\prime\prime}(x^{(1)}) \bigg(\left(\psi^{\prime}(\theta(x^{(1)}) + \theta(x^{(2)})) - \psi^{\prime}(\theta(x^{(1)}))\right)\left(\psi^{\prime}(\theta(x^{(1)}) + \theta(x^{(2)})) - \psi^{\prime}(\theta(x^{(2)}))\right) \\
      &\hspace{3cm} + \psi^{\prime\prime}(\theta(x^{(1)}) + \theta(x^{(2)}))\bigg) \\
  &\hspace{1cm} + \theta^{\prime}(x^{(1)})^2\bigg(\left(\psi^{\prime\prime}(\theta(x^{(1)}) + \theta(x^{(2)})) - \psi^{\prime\prime}(\theta(x^{(1)}))\right)\left(\psi^{\prime}(\theta(x^{(1)}) + \theta(x^{(2)})) - \psi^{\prime}(\theta(x^{(2)}))\right) \\
  &\hspace{3cm} + \left(\psi^{\prime}(\theta(x^{(1)}) + \theta(x^{(2)})) - \psi^{\prime}(\theta(x^{(1)}))\right)\left(\psi^{\prime\prime}(\theta(x^{(1)}) + \theta(x^{(2)}))\right) \\
  &\hspace{3cm} + \psi^{\prime\prime\prime}(\theta(x^{(1)}) + \theta(x^{(2)}))\bigg)\bigg). \numberthis
    \end{align*}
    The derivatives of $\theta(x)$ appearing in these formulas are
\begin{align}
  \theta^{\prime}(x) &= \frac{1}{\psi^{\prime\prime}(\theta(x))}, \\
  \theta^{\prime\prime}(x) &= -\frac{\psi^{\prime\prime\prime}(\theta(x))}{\psi^{\prime\prime}(\theta(x))^2} \theta^{\prime}(x) = -\frac{\psi^{\prime\prime\prime}(\theta(x))}{\psi^{\prime\prime}(\theta(x))^3}.
\end{align}
We then have
\begin{align}
  \theta(0) &= 0, \\
  \theta^{\prime}(0) &= \frac{1}{\sigma^2}, \\
  \theta^{\prime\prime}(0) &= - \frac{\psi^{\prime\prime\prime}(0)}{\sigma^6}.
\end{align}
Further, when we evaluate the derivatives of $R_{\sP}$ at $x^{(1)} = x^{(2)} = 0$, we will have many cancellations because, for $i \in \{1, 2\}$,
\begin{equation}
    \psi^{\prime}(\theta(x^{(1)}) + \theta(x^{(2)})) - \psi^{\prime}(\theta(x^{(i)})) = \psi^{\prime}(0) - \psi^{\prime}(0) = 0.
\end{equation}
In particular, for the Fisher information we find
\begin{equation}
    \frac{\partial^2 R_{\sP}}{\partial x^{(1)}\partial x^{(2)}}(0, 0) = \frac{1}{\sigma^2} \cdot \frac{1}{\sigma^2} \cdot (0 + \sigma^2) = \frac{1}{\sigma^2},
\end{equation}
as claimed, while for the third mixed derivatives we find
\begin{equation}
    \frac{\partial^3 R_{\sP}}{\partial x^{(1)^2}\partial x^{(2)}}(0, 0) = \frac{1}{\sigma^2} \cdot \left(0 - \frac{\psi^{\prime\prime\prime}(0)}{\sigma^6} \cdot (0 + \sigma^2) + \frac{1}{\sigma^4}\psi^{\prime\prime\prime}(0)\right) = 0,
\end{equation}
completing the proof.
\end{proof}

\subsection{Truncated exponentials: Proof of Proposition~\ref{prop:trunc-exp}}
\label{app:prop:trunc-exp}

We will use throughout the integral formula
    \begin{equation}
        \exp^{\leq D}(x) = \frac{1}{D!}\int_0^{\infty} \exp(-s) \left(x + s\right)^D ds,
        \label{eq:expD-int}
    \end{equation}
    which may be found in \cite{DCS-2003-TruncatedPolynomials}.

\begin{proof}[Proof of Proposition~\ref{prop:trunc-exp}]
    The first inequality of the first claim, $\exp^{\leq D}(x) > 0$, follows directly from this since $D$ is even, and the second inequality of the first claim, $\exp^{\leq D}(x) \leq \exp(|x|)$, follows by the triangle inequality and Taylor expansion.
    For the second claim, note that the left-hand inequality follows from the right-hand one: if the right-hand inequality holds, we may also bound
    \begin{equation}
        \exp^{\leq D}(x) = \exp^{\leq D}(x + y - y) \leq \exp^{\leq D}(x + y) \exp(100|y|),
    \end{equation}
    and rearranging gives the left-hand inequality.

    To prove the remaining right-hand inequality, we rewrite
    \begin{align*}
      \log \frac{\exp^{\leq D}(x + y)}{\exp^{\leq D}(x)}
      &= \log \exp^{\leq D}(x + y) - \log \exp^{\leq D}(x) \\
      &= \int_{x}^{x + y} \frac{d}{dt}\log \exp^{\leq D}(t) \, dt \\
      &= \int_x^{x + y} \frac{\exp^{\leq D - 1}(t)}{\exp^{\leq D}(t)}dt \\
      &= \int_x^{x + y} \left(1 - \frac{\frac{t^D}{D!}}{\exp^{\leq D}(t)}\right)dt \\
      &= y - \int_x^{x + y} \frac{t^D}{\int_0^{\infty} \exp(-s)(t + s)^Dds}dt \numberthis
    \end{align*}
    We will show that the quantities
    \begin{equation}
        I_D(t) \colonequals \frac{t^D}{\int_0^{\infty} \exp(-s)(t + s)^Dds}dt
    \end{equation}
    inside the remaining integral are uniformly bounded above for all $t \in \RR$ and $D \geq 0$ even.
    Fix some $t_0$ to be chosen later.
    We first note that $I_D(0) = 1$, and for $t > 0$ we have $I_D(t) \leq 1 / (\int_0^{\infty} \exp(-s)ds) = 1$.
    So, we may restrict our attention to $t < 0$.
    Let us write this as $I_D(-t)$ for $t > 0$.
    We have
    \begin{equation}
        \frac{1}{I_D(-t)} = \int_0^{\infty} \exp(-s)\left(1 - \frac{s}{t}\right)^Dds,
    \end{equation}
    and it suffices to bound this quantity from below.
    Set a further constant $C > 0$ to be fixed later.
    We consider three cases.

    \emph{Case 1, $t \leq t_0$:} We may bound
    \begin{align*}
      \frac{1}{I_D(-t)}
      &\geq \int_{2t}^{\infty} \exp(-s)\left(1 - \frac{s}{t}\right)^Dds \\
      &\geq \int_{2t}^{\infty} \exp(-s)ds \\
      &= 1 - \exp(-2t) \\
      &\geq 1 - \exp(-2t_0). \numberthis
    \end{align*}

    \emph{Case 2, $t > t_0$ and $D \leq Ct$:} We may bound using that $(1 - r)^D \geq 1 - Dr$,
    \begin{align*}
      \frac{1}{I_D(-t)}
      &\geq \int_0^{t/D} \exp(-s) \left(1 - \frac{Ds}{t}\right)ds \\
      &= \int_0^{t/D} \exp(-s)ds - \frac{D}{t}\int_0^{t/D} s\exp(-s)ds \\
      &= 1 - \exp\left(-\frac{t}{D}\right) - \frac{D}{t}\left(1 - \exp\left(-\frac{t}{D}\right)\left(\frac{t}{D} + 1\right)\right)
        \intertext{This is an increasing function of $t / D$, on $t / D \geq 0$, increasing from 0 to 1 as $t / D \to \infty$.
        Thus, since under the assumption of this case $t/D \geq 1/C$, our expression will be bounded below as}
      &\geq 1 - \exp\left(\frac{1}{C}\right) - C\left(1 - \exp\left(-\frac{1}{C}\right)\left(\frac{1}{C} + 1\right)\right). \numberthis
    \end{align*}

    \emph{Case 3, $t > t_0$ and $D > Ct$:}
    \begin{align*}
      \frac{1}{I_D(-t)}
      &\geq \int_{2t}^{\infty}\exp(-s)\left(\frac{s}{t} - 1\right)^Dds \\
      &= \exp(-2t) \int_0^{\infty}\exp(-s) \left(1 + \frac{s}{t}\right)^Dds \\
      &\geq \exp(-2t) \int_0^{\infty}\exp(-s) \left(1 + \frac{s}{t}\right)^{Ct} ds \\
      &\geq \exp(-2t) \int_t^{t + 1}\exp(-s) \left(1 + \frac{s}{t}\right)^{Ct} ds \\
      &\geq \exp(-2t) \cdot \exp(-t - 1) 2^{Ct} \\
      &= \exp\left(t(C\log(2) - 3) - 1\right)
        \intertext{so, so long as $C > 3 / \log(2)$, we have}
      &\geq \exp\left(t_0(C\log(2) - 3) - 1\right). \numberthis
    \end{align*}

    Concretely, taking $t_0 = C = 5$ and evaluating our lower bounds in the three cases gives that, for all $D \geq 0$ even and $t \in \RR$, $I_D(t) \leq 99$.
    Plugging into our original calculations, we find that
    \begin{equation}
        \log \frac{\exp^{\leq D}(x + y)}{\exp^{\leq D}(x)} \leq 100 |y|,
    \end{equation}
    and the result follows.
\end{proof}

\subsection{Binomial coefficients: Proof of Proposition~\ref{prop:binom-lb}}
\label{app:pf:prob:binom-lb}

\begin{proof}[Proof of Proposition~\ref{prop:binom-lb}]
    We first expand directly:
    \begin{align*}
  \binom{k}{t}
  &= \frac{1}{t!} k(k - 1) \cdots (k - t + 1) \\
  &= \frac{k^t}{t!} 1 \left(1 - \frac{1}{k}\right) \cdots \left(1 - \frac{t - 1}{k}\right) \\
      &= \frac{k^t}{t!}\exp\left(\sum_{i = 0}^{t - 1}\log\left(1 - \frac{i}{k}\right)\right)
        \intertext{and for the remaining sum, note that $\log(1 - x) \geq -2x$ for all $0 \leq x \leq \frac{1}{2}$, so}
      &\geq \frac{k^t}{t!}\exp\left(-\frac{2}{k}\sum_{i = 0}^{t - 1}i\right) \\
      &\geq \frac{k^t}{t!}\exp\left(-\frac{t^2}{k}\right), \numberthis
    \end{align*}
    as claimed.
\end{proof}

\subsection{Polynomial advantage in spiked tensor model: Proof of Proposition~\ref{prop:kwb-stm-2}}
\label{app:pf:prop:kwb-stm-2}

\begin{proof}[Proof of Proposition~\ref{prop:kwb-stm-2}]
    We give the argument for the case $\pi = \Unif(\{\pm 1\})$, which is easily generalized to other bounded distributions for the prior entries.
    By the proof in Section~3.1.1 of \cite{KWB-2022-LowDegreeNotes}, we have
    \begin{align*}
      \Adv_{\leq D - 2}(\sX_n, \sP)
      &\leq \sum_{d = 0}^{D - 2} \left(\lambda^2 q^{q/2 + 1} (2D)^{(q - 2)/2}\right)^d
        \intertext{By taking $b_q$ large enough, since we assume $\lambda \geq b_q D^{-(q - 2)/4}$, we may ensure the quantity being raised to powers is greater than 2. Thus,}
      &\leq 2\left(\lambda^2 q^{q/2 + 1} (2D)^{(q - 2)/2}\right)^{D - 2}. \numberthis
    \end{align*}
    On the other hand, by the proof in Section 3.1.2 of \cite{KWB-2022-LowDegreeNotes}, we have
    \begin{equation}
      \Adv_{\leq D}(\sX_n, \sP) \geq \left(\lambda^2 e^{-q} q^{q/2} D^{(q - 2)/2}\right)^D,
    \end{equation}
    and combining the two inequalities gives the result upon taking $b_q$ large enough.
\end{proof}

\end{document}